\documentclass[11pt]{amsart}
\usepackage{amsmath,amssymb,amsthm,mathrsfs}
\usepackage[margin=1.15in]{geometry}
\usepackage[hidelinks]{hyperref}
\hypersetup{
  pdftitle={Subgroup measures and profinite rigidity in free groups},
  pdfauthor={Shrinit Singh},
  pdfkeywords={word measures, subgroup measures, profinite rigidity, free groups, free products}
}

\theoremstyle{plain}
\newtheorem{theorem}{Theorem}[section]
\newtheorem{proposition}[theorem]{Proposition}
\newtheorem{lemma}[theorem]{Lemma}
\newtheorem{corollary}[theorem]{Corollary}
\theoremstyle{definition}
\newtheorem{definition}[theorem]{Definition}
\newtheorem{remark}[theorem]{Remark}
\newtheorem{example}[theorem]{Example}
\newtheorem{question}[theorem]{Question}

\newcommand{\acl}{\operatorname{acl}}
\newcommand{\rk}{\operatorname{rk}}
\newcommand{\Hom}{\operatorname{Hom}}
\newcommand{\Epi}{\operatorname{Epi}}
\newcommand{\Aut}{\operatorname{Aut}}
\newcommand{\Stab}{\operatorname{Stab}}
\newcommand{\ImEpi}{\operatorname{ImEpi}}
\newcommand{\MP}{\operatorname{MP}}
\newcommand{\Zhat}{\widehat{\mathbb{Z}}}
\newcommand{\Fhat}{\widehat{F}}
\newcommand{\fii}{\trianglelefteq_{\mathrm{f.i.}}}

\title[Subgroup measures and profinite rigidity]
{Subgroup measures and profinite rigidity in free groups}
\author{Shrinit Singh}
\address{International Centre for Theoretical Sciences, Tata Institute of
Fundamental Research, Bengaluru 560089, India}
\email{shrinit.singh@icts.res.in}

\begin{document}

\begin{abstract}
Let $H$ be a finitely generated subgroup of a finitely generated group $F$.
Restriction of a uniformly random homomorphism $F\to G$, with $G$ finite,
defines a probability measure on $\Hom(H,G)$; for cyclic $H$ this is the usual
word measure.  We introduce marked and setwise notions of profinite rigidity
for these subgroup measures.  In a finitely generated LERF group, an
isomorphism between finitely generated subgroups preserves all induced
measures if and only if its profinite completion extends to an automorphism of
the ambient profinite completion.  This gives a discrete--profinite orbit
criterion and identifies the precise obstruction separating marked from
setwise rigidity.  For free groups, both notions are invariant under adjoining
a free factor.  We classify the measure classes of subgroups of $F_n$
containing $[F_n,F_n]$: they are the $\Aut(F_n)$-orbits determined by Smith
invariants, and the subgroups determined uniquely by their measures are
exactly $F_n^{\,m}[F_n,F_n]$.  Finally, marked rigidity is preserved and
reflected by free products compatible with an ambient free-product
decomposition.  Consequently, for every $n\ge2$ and $1\le r\le n$, there is
an infinite-index profinitely rigid subgroup of $F_n$ of rank $r$ whose
algebraic closure is $F_n$.
\end{abstract}

\maketitle

\section{Introduction}\label{sec:intro}

Word maps form a classical interface between combinatorial group theory, finite
group theory, algebraic groups, and profinite methods. Traditionally one asks
which elements occur as values of a given word. A finer invariant is obtained
by retaining the full distribution of those values: for a finite group, one
evaluates the word on a uniformly random tuple and records the resulting
probability measure. These \emph{word measures} can detect algebraic features
that are invisible from the image of the word map alone, and they are naturally
suited to questions governed by finite quotients.

\subsection{Word maps and their induced measures}\label{sec:words}

Let $F_n=\langle x_1,\dots,x_n\rangle$ be a free group on $x_1,\dots,x_n$. For a word
\[
  w=\prod_{j=1}^{m}x_{i_j}^{\varepsilon_j}\in F_n,
  \qquad i_j\in\{1,\dots,n\},\ \varepsilon_j\in\mathbb Z\setminus\{0\},
\]
the induced word map on a group $G$ is
\[
  w_G\colon G^{n}\longrightarrow G,\qquad
  w_G(g_1,\dots,g_n)=\prod_{j=1}^{m}g_{i_j}^{\varepsilon_j}.
\]
When $G$ is finite, the uniform measure on $G^{n}$ pushes forward along $w_G$
to a probability measure $\mu_w$ on $G$: for $S\subseteq G$,
\[
  \mu_w(S)=\Pr\nolimits_w(S)=\frac{\bigl|w_G^{-1}(S)\bigr|}{|G|^{n}} .
\]
In particular $\mu_w(g)$ is the probability that $w_G$ evaluates to $g\in G$
on uniformly and independently chosen inputs.

The conjecture of Amit--Vishne and Shalev \cite{AmitVishne,Shalev} asserts that
this family of measures determines a word up to the natural action of
$\Aut(F_n)$: if $u,w\in F_n$ induce the same measure on every finite group,
then $u$ and $w$ should lie in the same $\Aut(F_n)$-orbit. Hanany, Meiri, and
Puder \cite{HMP} related this problem to automorphism orbits in the profinite
completion, motivating the terminology \emph{profinite rigidity}.
\begin{definition}\label{def:word-rigid}
A word $w\in F_n$ is \emph{profinitely rigid} if every word $u\in F_n$ that
induces the same probability measure as $w$ on every finite group lies in the
same $\Aut(F_n)$-orbit as $w$.  The word $w$ is \emph{universally profinitely
rigid} if it is profinitely rigid in $F_n*P$ for every finitely generated free
group $P$.
\end{definition}
The conjecture asserts that every word in a finitely generated free group is
profinitely rigid. Jaikin-Zapirain, Souza, and Zalesskii \cite{JSZ} proved that
profinite rigidity of words implies their universal profinite rigidity; we
generalise this to finitely generated subgroups in Theorem~\ref{thm:universal}.

Several results support this picture. In particular, primitive words and free
factors are detected by finite-group statistics \cite{PP,Wilton, GJ}, and further
families of rigid words were obtained in \cite{HMP,Wilton2,AF}. The purpose
of this paper is to extend the measure-theoretic framework from individual
words to arbitrary finitely generated subgroups of free groups.

\subsection{Measures induced by subgroups: two generalisations}\label{sec:two-notions}

The subgroup analogue is obtained by replacing evaluation by restriction. Let
$P$ be a finitely generated group and let $H\le P$ be finitely generated. For
any group $K$, there is a restriction map
\[
  r\colon\Hom(P,K)\longrightarrow\Hom(H,K),\qquad r(f)=f|_H .
\]
When $K$ is finite, pushing forward the uniform distribution on $\Hom(P,K)$
along $r$ gives a probability measure $\mu^P_{H,K}$ on $\Hom(H,K)$. Its mass
at $h\in\Hom(H,K)$ is
\[
  \Pr\nolimits^{P}_{H,K}(h)
  :=\mu^P_{H,K}(h)
  =\frac{\bigl|\{f\in\Hom(P,K)\mid f|_H=h\}\bigr|}
          {|\Hom(P,K)|}.
\]
When the ambient group is fixed or clear from context, we suppress the
superscript $P$. If $P=F_n$ and $H=\langle w\rangle$ is cyclic, evaluation at
$w$ identifies $\mu^P_{H,K}$ with the usual word measure $\mu_w$.

Let $H_1,H_2\le P$ be finitely generated subgroups and let
$\alpha\colon H_1\to H_2$ be an isomorphism. We say that $H_1$ and $H_2$ are
\emph{measure equivalent via $\alpha$} if, for every finite group $K$ and every
$h\in\Hom(H_2,K)$,
\begin{equation}\label{eq:meq}
  \Pr\nolimits_{H_2,K}(h)=\Pr\nolimits_{H_1,K}(h\circ\alpha).
\end{equation}
Equivalently, after identifying $\Hom(H_1,K)$ with $\Hom(H_2,K)$ via $\alpha$,
the pushforward distributions coincide for every finite $K$.

For cyclic subgroups, word rigidity remembers not only the subgroup but also its
distinguished generator. For a general finitely generated subgroup, there are
therefore two natural levels of rigidity. In the marked version, the specific
measure equivalence $\alpha$ must be realised by an ambient automorphism. In
the setwise version, one asks only that the two subgroups lie in the same
ambient automorphism orbit.
\begin{definition}\label{def:rigid}
Let $G$ be a finitely generated group and let $H\le G$ be finitely generated.
We say that $H$ is \emph{profinitely rigid in $G$} if, whenever $L\le G$ is
finitely generated and $\alpha\colon H\to L$ is a measure equivalence, there
exists $\varphi\in\Aut(G)$ such that
$\varphi(H)=L$ and $\varphi|_H=\alpha$.
\end{definition}
\begin{definition}\label{def:setwise}
Let $G$ be a finitely generated group.  A finitely generated subgroup $H\le G$ is
\emph{setwise profinitely rigid in $G$} if, whenever $L\le G$ is finitely
generated and measure equivalent to $H$, the subgroup $L$ lies in the
$\Aut(G)$-orbit of $H$.
\end{definition}
When the distinction matters, we refer to Definition~\ref{def:rigid} as
\emph{marked profinite rigidity}. Marked rigidity plainly implies setwise
rigidity; Theorem~\ref{thm:setwise-marked} identifies the precise additional
condition. Both notions depend on the embedding $H\le G$, not merely on the
abstract group $H$. We therefore also record the version stable under free
extension.
\begin{definition}\label{def:universal-rigid}
Let $F$ be a finitely generated free group and $H\le F$ finitely generated. We
call $H$ \emph{universally profinitely rigid} (respectively, \emph{universally
setwise profinitely rigid}) if $H$ is profinitely rigid (respectively, setwise
profinitely rigid) in $F*P$ for every finitely generated free group $P$.
\end{definition}
Thus ``universal'' refers only to free extensions of the fixed ambient group.
It does not mean rigidity in every finitely generated free group containing
$H$.

Ambient automorphisms preserve the induced measures
(Lemma~\ref{lem:aut}); the rigidity problem asks for converses to this
observation. The cyclic case recovers the original word problem.
\begin{corollary}\label{cor:word-cyclic}
A word $w\in F_n$ is profinitely rigid in the sense of
Definition~\ref{def:word-rigid} if and only if the cyclic subgroup
$\langle w\rangle$ is profinitely rigid in $F_n$ in the sense of
Definition~\ref{def:rigid}.
\end{corollary}
\begin{proof}
If $w=1$, both sides hold trivially: $\mu_1$ is the point mass at the identity
in every finite group, so $\mu_u=\mu_1$ forces $u=1$ by residual finiteness,
and $\langle1\rangle$ is the only subgroup isomorphic to itself. Assume
$w\neq1$, so $\langle w\rangle\cong\mathbb Z$.

For $u\in F_n\setminus\{1\}$, let $\alpha\colon\langle w\rangle\to\langle
u\rangle$ be the isomorphism with $\alpha(w)=u$. Evaluation at $u$ (resp.\ $w$)
identifies $\Hom(\langle u\rangle,K)$ (resp.\ $\Hom(\langle w\rangle,K)$) with
$K$; writing $h_k$ for $u\mapsto k$,
\[
  \Pr\nolimits_{\langle u\rangle,K}(h_k)=\mu_u(k),
  \qquad
  \Pr\nolimits_{\langle w\rangle,K}(h_k\circ\alpha)=\mu_w(k),
\]
so $\alpha$ is a measure equivalence iff $\mu_w=\mu_u$ on every finite group.

Suppose $w$ is a profinitely rigid word and $\alpha\colon\langle w\rangle\to L$
is a measure equivalence. Then $L=\langle u\rangle$ with $u=\alpha(w)\neq1$ and,
by the above, $\mu_w=\mu_u$; rigidity of $w$ gives $\varphi\in\Aut(F_n)$ with
$\varphi(w)=u$, hence $\varphi|_{\langle w\rangle}=\alpha$ and
$\varphi(\langle w\rangle)=L$. Thus $\langle w\rangle$ is profinitely rigid.

Conversely, suppose $\langle w\rangle$ is profinitely rigid and $u$ induces the
same measures as $w$. Then $u\neq1$ (since $w\neq1$ is detected in some finite
quotient), and $\alpha\colon w\mapsto u$ is a measure equivalence
$\langle w\rangle\to\langle u\rangle$. Rigidity supplies $\varphi\in\Aut(F_n)$
with $\varphi|_{\langle w\rangle}=\alpha$, in particular $\varphi(w)=u$. Thus
$w$ is a profinitely rigid word.
\end{proof}
In this sense, Shalev's conjecture is the rank-one instance of subgroup
profinite rigidity.

\subsection{Main results}\label{sec:main-results}

The main results concern the relation between marked and setwise rigidity, the
behaviour of rigidity under free extension and free products, and explicit
families of rigid subgroups.

\medskip\noindent\textbf{A profinite orbit principle.}
The measure relation itself has an exact profinite interpretation.  If $G$ is
finitely generated and LERF, $H_1,H_2\le G$ are finitely generated, and
$\alpha\colon H_1\to H_2$ is an isomorphism, then $\alpha$ is a measure
equivalence if and only if
\[
  \widehat\alpha\colon\overline{H_1}\longrightarrow\overline{H_2}
\]
extends to an automorphism of $\widehat G$ (Proposition~\ref{prop:lerf}).  In
particular, for finitely generated subgroups of free groups, measure classes are
exactly the intersections of suitable $\Aut(\widehat F)$-orbits with the
discrete group.  Corollary~\ref{cor:orbit-intersection} makes this precise and
recasts both rigidity notions as discrete--profinite orbit-intersection
properties.

\medskip\noindent\textbf{Comparing the two notions.}
For finitely generated subgroup $H\le G$,  write
\[
  \MP_G(H):=\{\gamma\in\Aut(H)\ :\ \gamma\ \text{is a measure equivalence }H\to H\}
\]
for the group of \emph{measure-preserving automorphisms of $H$ relative to
$G$}, and let
\[
  \rho_H\colon \Stab_{\Aut(G)}(H)\longrightarrow\Aut(H),\qquad
  \psi\longmapsto\psi|_H,
\]
be the restriction homomorphism defined on the setwise stabiliser of $H$ in
$\Aut(G)$. One always has $\operatorname{im}\rho_H\subseteq\MP_G(H)$
(Lemma~\ref{lem:aut}).
\begin{theorem}\label{thm:setwise-marked}
Let $G$ be a finitely generated group and let $H\le G$ be finitely generated.
Then $H$ is profinitely rigid in $G$ if and only if
\begin{enumerate}
\item[(i)] $H$ is setwise profinitely rigid in $G$, and
\item[(ii)] $\MP_G(H)=\operatorname{im}\rho_H$.
\end{enumerate}
\end{theorem}
Thus the only additional information carried by marked rigidity is the
extendability of measure-preserving automorphisms of $H$. For free factors this
condition is automatic.
\begin{corollary}\label{cor:ff-setwise-marked}
If $H$ is a free factor of a finitely generated free group $F$, then $H$ is
profinitely rigid in $F$ if and only if it is setwise profinitely rigid in $F$.
\end{corollary}

\medskip\noindent\textbf{Reduction to tuples.}
Marked subgroup rigidity admits an equivalent formulation in terms of generating
tuples; this requires only residual finiteness.
\begin{theorem}\label{thm:tuples}
Let $\Gamma$ be a finitely generated residually finite group and let
$H\le\Gamma$ be finitely generated. Then $H$ is profinitely rigid in $\Gamma$
if and only if some (equivalently, every) ordered generating tuple of $H$ is
profinitely rigid as a tuple in $\Gamma$.
\end{theorem}
Theorem~\ref{thm:tuples} identifies marked subgroup rigidity with the
corresponding tuple problem and places word rigidity as its rank-one case.

\medskip\noindent\textbf{Universality and reduction.}
Both notions are stable under adjoining a free factor.
\begin{theorem}\label{thm:universal}
Let $F$ and $P$ be finitely generated free groups and let $K\le F$ be finitely
generated. Then
\begin{enumerate}
\item[(i)] $K$ is profinitely rigid in $F$ if and only if it is profinitely
  rigid in $F*P$;
\item[(ii)] $K$ is setwise profinitely rigid in $F$ if and only if it is setwise
  profinitely rigid in $F*P$.
\end{enumerate}
Consequently, marked and setwise profinite rigidity are stable under free
extension, and either property for $K\le F$ is equivalent to the corresponding
property in $\acl_F(K)$.
\end{theorem}
The same theorem also permits reduction to the algebraic closure
$\acl_F(K)$; we isolate this consequence as Proposition~\ref{prop:reduce}.

\medskip\noindent\textbf{Setwise rigid subgroups.}
For setwise rigidity we obtain a particularly strong source of examples. We call
$H\le F$ \emph{measure-detected} if every finitely generated subgroup measure
equivalent to $H$ is equal to $H$. Finite-index verbal subgroups have this
property (Theorem~\ref{thm:measure-detected}). More generally, for subgroups of
$F_n$ containing $[F_n,F_n]$, the measure class is exactly the
$\Aut(F_n)$-orbit and is classified by the invariant factors of the image in
$\mathbb Z^n$ (Theorem~\ref{thm:above-commutator}).

\medskip\noindent\textbf{Rigid subgroups.}
For marked rigidity, the principal structural result is a free-product theorem.
\begin{theorem}\label{thm:free-prod}
Let $F$ be a finitely generated free group, let $F=A_1*\dots*A_k$ be an internal
free product decomposition, let $H_i\le A_i$ be finitely generated subgroups, and
put
\[
  H=H_1*\dots*H_k .
\]
Then $H$ is profinitely rigid in $F$ \emph{if and only if} every $H_i$ is
profinitely rigid in $F$. In that case $H$ is profinitely rigid in every free
extension $F*P$ with $P$ finitely generated free.
\end{theorem}
The corresponding forward implication holds for setwise rigidity
(Proposition~\ref{prop:free-prod-setwise}); we do not assert a converse in the
setwise setting.  A useful consequence is the assembly principle
(Corollary~\ref{cor:assembly}): if profinitely rigid words are chosen in distinct
free factors, then the subgroup freely generated by those words is profinitely
rigid.  Thus every new family of rigid words automatically produces families of
rigid subgroups of arbitrary rank.
\begin{corollary}\label{cor:powers}
Let $F_n=\langle x_1,\dots,x_n\rangle$. For $r\le n$ and nonzero integers
$m_1,\dots,m_r$, the subgroup
$\langle x_1^{m_1},x_2^{m_2},\dots,x_r^{m_r}\rangle$ is profinitely rigid in
$F_n$.
\end{corollary}
Combining Theorem~\ref{thm:free-prod} with known rigid words produces rigid
subgroups well beyond the free-factor case. In particular, for every $n\ge2$
and $1\le r\le n$, there is a profinitely rigid subgroup $H\le F_n$ of rank
$r$ and infinite index with $\acl_{F_n}(H)=F_n$
(Corollary~\ref{cor:ranks}).

We also consider whether subgroup rigidity can be used in the opposite
direction to produce new rigid words. This leads to a transitivity question
for chains $K\le H\le F$ (Question~\ref{q:transitivity}). We prove the desired
transitivity whenever $\acl_F(K)\le H$ (Proposition~\ref{prop:acl-intermediate})
and record the new word families that would follow from a general positive
answer (Corollary~\ref{cor:conditional-surface}).

Our arguments build on the measure-preserving and separability results for
primitive and surface words in \cite{PP,Wilton,Wilton2,AF,HMP}, together with
the profinite criterion of Garrido and Jaikin-Zapirain \cite{GJ}. The subgroup
formalism developed here provides a common setting in which these inputs can be
combined.

\subsection*{Organisation of the paper}

Section~\ref{sec:intro} sets up the two notions of subgroup profinite rigidity
and states the main results; Section~\ref{sec:prelim} collects the tools---
LERF, verbal subgroups, algebraic closure, and free profinite products---used
throughout. Section~\ref{sec:setwise-vs-marked} compares the two notions: we
introduce $\MP_G(H)$ and $\rho_H$ and prove Theorem~\ref{thm:setwise-marked}
(\S\ref{sec:mp}), isolate the invariants visible in the induced measures
(\S\ref{sec:invariants}), and prove Theorem~\ref{thm:universal} together with
the reduction principle (\S\ref{sec:universal}). Section~\ref{sec:setwise-rigid}
is devoted to setwise profinitely rigid subgroups: finite-index verbal
subgroups (\S\ref{sec:verbal}) and subgroups above the commutator subgroup
(\S\ref{sec:commutator}). Section~\ref{sec:freeproducts} treats profinitely
rigid subgroups and proves Theorem~\ref{thm:free-prod}, from which
Corollary~\ref{cor:powers} and the new families of
Corollaries~\ref{cor:examples} and \ref{cor:ranks} follow;
\S\ref{sec:back-to-words} records the transitivity question governing the
reverse passage from rigid subgroups to rigid words and proves the algebraic-closure case. Finally, Section~\ref{sec:tuples} develops profinite rigidity for
tuples and proves Theorem~\ref{thm:tuples}.

\section{Preliminaries}\label{sec:prelim}

We collect the profinite and free-group tools used throughout the paper. The
first subsection recalls profinite completions and subgroup separability; the
second records the facts about verbal subgroups, algebraic closure, and free
profinite products needed later.

\subsection{Profinite completions and LERF}\label{sec:lerf}

The profinite completion packages all finite quotients of a group into a single
topological object and will mediate between equality of finite-group measures
and automorphisms of the completion.

Let $P$ be a group. The profinite topology on $P$ has as a basis the left
cosets of finite-index subgroups of $P$. The profinite completion $\widehat P$
of $P$ is the inverse limit of the finite quotients of $P$: writing
$N\fii P$ to mean that $N$ is a normal subgroup of finite index,
\[
  \widehat P=\varprojlim_{N\fii P} P/N .
\]
Equipped with the inverse limit topology, in which each $P/N$ carries the
discrete topology, $\widehat P$ is a compact, Hausdorff, totally disconnected
topological group. There is a natural homomorphism
\[
  \iota\colon P\longrightarrow\widehat P,\qquad \iota(\gamma)=(\gamma N)_{N\fii P}.
\]
A group $G$ is \emph{LERF} (locally extended residually finite, or subgroup
separable) if any finitely generated subgroup $H\le G$ and any
$g\in G\setminus H$ can be separated in a finite quotient, i.e.\ there is a
finite-index normal subgroup $N\fii G$ with $gN\notin HN/N$. Finitely generated
free groups are LERF by M.~Hall's theorem \cite{Hall}.

We shall use repeatedly the following criterion of Garrido and
Jaikin-Zapirain.
\begin{lemma}[{\cite[Proposition 3.2]{GJ}}]\label{lem:gj}
Let $\mathcal G$ be a finitely generated profinite group and let
$\alpha\colon \mathcal H_1\to\mathcal H_2$ be an isomorphism between two closed
subgroups of $\mathcal G$. The following are equivalent.
\begin{enumerate}
\item There exists $\widetilde\alpha\in\Aut(\mathcal G)$ with
      $\widetilde\alpha|_{\mathcal H_1}=\alpha$.
\item For every finite group $K$ and every $\beta\in\Hom(\mathcal H_2,K)$,
      $\Pr_{\mathcal H_1,K}(\beta\circ\alpha)=\Pr_{\mathcal H_2,K}(\beta)$.
\end{enumerate}
\end{lemma}
For finitely generated subgroups of a LERF group, Lemma~\ref{lem:gj} translates
the equality of the discrete induced measures into an automorphism of the
ambient profinite completion. The point is that LERF identifies
$\widehat{H_i}$ with the closure $\overline{H_i}\le\widehat G$.
\begin{proposition}[Profinite orbit principle]\label{prop:lerf}
Let $G$ be a finitely generated LERF group, let $H_1,H_2\le G$ be finitely
generated subgroups, and let $\alpha\colon H_1\to H_2$ be an isomorphism.  The
following are equivalent.
\begin{enumerate}
\item[(i)] $\alpha$ is a measure equivalence relative to $G$;
\item[(ii)] the induced continuous isomorphism
\[
  \widehat\alpha\colon\overline{H_1}\longrightarrow\overline{H_2}
\]
extends to an automorphism of $\widehat G$.
\end{enumerate}
\end{proposition}
\begin{proof}
Since $G$ is LERF and the $H_i$ are finitely generated, the embeddings
$H_i\hookrightarrow G$ induce canonical isomorphisms
$\widehat{H_i}\cong\overline{H_i}$.  For every finite group $K$, restriction
along the dense embeddings gives natural bijections
\[
  \Hom(\widehat G,K)\cong\Hom(G,K),\qquad
  \Hom(\overline{H_i},K)\cong\Hom(H_i,K).
\]
These bijections identify the uniform measure on $\Hom(\widehat G,K)$ with the
uniform measure on $\Hom(G,K)$ and therefore identify the induced profinite
measures on $\Hom(\overline{H_i},K)$ with the discrete measures
$\mu^G_{H_i,K}$.  Under these identifications, condition~(i) is exactly
condition~(2) of Lemma~\ref{lem:gj} for the isomorphism $\widehat\alpha$.
Lemma~\ref{lem:gj} therefore gives the equivalence with~(ii).
\end{proof}

The proposition turns measure equivalence into an orbit problem in the ambient
profinite completion.  The following formulation is particularly useful.
\begin{corollary}\label{cor:orbit-intersection}
Let $G$ be a finitely generated LERF group and let $H\le G$ be finitely
generated.
\begin{enumerate}
\item[(i)] $H$ is profinitely rigid in $G$ if and only if, for every
$\Theta\in\Aut(\widehat G)$ satisfying $\Theta(H)\le G$, there exists
$\varphi\in\Aut(G)$ such that
\[
  \varphi|_H=\Theta|_H.
\]
\item[(ii)] $H$ is setwise profinitely rigid in $G$ if and only if, for every
$\Theta\in\Aut(\widehat G)$ satisfying $\Theta(H)\le G$, one has
\[
  \Theta(H)\in\Aut(G)\cdot H.
\]
\end{enumerate}
\end{corollary}
\begin{proof}
Suppose first that $\Theta\in\Aut(\widehat G)$ satisfies $L:=\Theta(H)\le G$.
Then $\alpha:=\Theta|_H\colon H\to L$ is an isomorphism, and continuity shows
that its profinite completion is $\Theta|_{\overline H}$.  By
Proposition~\ref{prop:lerf}, $\alpha$ is a measure equivalence.  Marked,
respectively setwise, rigidity of $H$ gives the conclusions in (i),
respectively (ii).

Conversely, let $\alpha\colon H\to L$ be a measure equivalence with $L\le G$
finitely generated.  Proposition~\ref{prop:lerf} provides
$\Theta\in\Aut(\widehat G)$ extending $\widehat\alpha$.  In particular
$\Theta(H)=\alpha(H)=L\le G$.  The stated orbit-intersection property then
gives the required ambient automorphism in the marked case, or the required
$\Aut(G)$-orbit equality in the setwise case.
\end{proof}

Equivalently, the setwise measure class of $H$ is exactly the set of discrete
returns of its profinite automorphism orbit:
\begin{equation}\label{eq:measure-class-return}
  \{L\le G:L\text{ is measure equivalent to }H\}
  =\{\Theta(H):\Theta\in\Aut(\widehat G),\ \Theta(H)\le G\}.
\end{equation}
Thus setwise rigidity asks whether all such returns lie in one
$\Aut(G)$-orbit, while marked rigidity asks in addition that the returning
profinite automorphism be realizable on $H$ by a discrete ambient
automorphism.

Taking $H_1=H_2=H$ in Proposition~\ref{prop:lerf} also gives the useful
identity
\begin{equation}\label{eq:MP-profinite}
  \MP_G(H)=\{\gamma\in\Aut(H):\widehat\gamma\text{ extends to an element of }
  \Aut(\widehat G)\}
\end{equation}
whenever $G$ is finitely generated and LERF.

\begin{remark}\label{rmk:winverse}
A word $w$ and its inverse $w^{-1}$ can induce different measures on finite
groups even though $\langle w\rangle=\langle w^{-1}\rangle$. In that situation
Proposition~\ref{prop:lerf} only fixes $\overline{\langle w\rangle}$ up to
$\Aut(\widehat{F_n})$; it does not produce an automorphism of $F_n$ taking $w$
to $w^{-1}$. 
\end{remark}
We also use the characterisation of free factors by uniform induced measures.
\begin{proposition}\cite[Parzanchevski--Puder, restated]{PP}\label{prop:uniform}
Let $H\le F_n$ be finitely generated. Then $\mu_{H,G}$ is the uniform measure on
$\Hom(H,G)$ for every finite group $G$ if and only if $H$ is a free factor of
$F_n$.
\end{proposition}

\subsection{Verbal subgroups and algebraic closure}\label{sec:prelim-tools}

For a set of words $V$ and a group $P$, let $V(P)$ denote the subgroup generated
by all values of words from $V$ on $P$. This subgroup is fully invariant. For
the profinite group $\widehat{F_n}$, we use $V(\widehat{F_n})$ for the
\emph{closed} subgroup topologically generated by these values.
\begin{lemma}\label{lem:verbal-closure}
For every set of words $V$, one has
$V(\widehat{F_n})=\overline{V(F_n)}$.
\end{lemma}
\begin{proof}
Since $V(F_n)\subseteq V(\widehat{F_n})$ and the latter is closed,
$\overline{V(F_n)}\subseteq V(\widehat{F_n})$. Conversely, each word map
$v\colon\widehat{F_n}^{\,m}\to\widehat{F_n}$ is continuous and $F_n$ is dense in
$\widehat{F_n}$, so every value $v(\hat g_1,\dots,\hat g_m)$ lies in
$\overline{v(F_n^{\,m})}\subseteq\overline{V(F_n)}$. Hence the topological
generators of $V(\widehat{F_n})$ lie in $\overline{V(F_n)}$, giving the reverse
inclusion.
\end{proof}
\begin{proposition}\label{prop:verbal-invariant}
Let $V$ be any set of words. Then $\overline{V(F_n)}$ is invariant under every
continuous automorphism of $\widehat{F_n}$.
\end{proposition}
\begin{proof}
By Lemma~\ref{lem:verbal-closure}, $\overline{V(F_n)}=V(\widehat{F_n})$. For any
$\Theta\in\Aut(\widehat{F_n})$ the identity
\[
  \Theta\bigl(v(\hat g_1,\dots,\hat g_m)\bigr)=v(\Theta\hat g_1,\dots,\Theta\hat g_m)
\]
shows that $\Theta$ carries each value of $V$ to another such value, so it
preserves the abstract subgroup they generate, and being a homeomorphism, its
closure; applying this to $\Theta^{-1}$ gives equality.
\end{proof}
We use the free-factor notion of algebraic closure from
\cite[Section 3.4]{MVW}.
\begin{definition}\label{def:acl}
Let $K\le F$ be a finitely generated subgroup of a free group. We say $K$ is
\emph{algebraic} in $F$ if it is contained in no proper free factor of $F$. For
every finitely generated $K\le F$ there is a unique free factor $\acl_F(K)$
containing $K$ in which $K$ is algebraic; it is the smallest free factor of $F$
containing $K$, and we call it the \emph{algebraic closure} of $K$ in $F$.
Equivalently,
\[
  \acl_F(K)=\bigcap\,\{A\le F: A\text{ is a free factor of }F,\ K\le A\},
\]
the intersection being a free factor because a smallest member of the family
exists and equals it.
\end{definition}
The next lemma records the ambient-independence of algebraic closure along free
factor inclusions. The free-factor hypothesis is essential.
\begin{lemma}\label{lem:acl-trans}
Let $K\le P\le F$ with $F$ free of finite rank, $P$ a free factor of $F$ and $K$
finitely generated. Then $\acl_F(K)=\acl_P(K)$.
\end{lemma}
\begin{proof}
A free factor of a free factor is a free factor: if $F=P*R$ and $P=Q*S$ then $F=Q*S*R$. Hence $\acl_P(K)$ is a free factor of $F$ containing $K$, and minimality gives $\acl_F(K)\le\acl_P(K)\le P$.

For the reverse inclusion, put $A=\acl_F(K)$ and write $F=A*R$. Since $A\le P$, the Kurosh decomposition of $P$ relative to $F=A*R$ contains $P\cap A=A$ as a free factor. Thus $A$ is a free factor of $P$. Because $K\le A$, the minimality of $\acl_P(K)$ gives $\acl_P(K)\le A=\acl_F(K)$, and equality follows.
\end{proof}
We shall also use monotonicity: if $K\le K'\le F$, then
$\acl_F(K)\le\acl_F(K')$. Along a compatible free product decomposition,
algebraic closure behaves factorwise.
\begin{lemma}\label{lem:acl-freeprod}
Let $F=A_1*\dots*A_k$ be an internal free product decomposition of a finitely
generated free group, let $H_i\le A_i$ be finitely generated, and put
$H=H_1*\dots*H_k$. Then
\[
  \acl_F(H)=\acl_F(H_1)*\dots*\acl_F(H_k),
\]
an internal free product; in particular $H$ is algebraic in $F$ if and only if
each $H_i$ is algebraic in $A_i$.
\end{lemma}
\begin{proof}
Put $A_i'=\acl_F(H_i)$, which equals $\acl_{A_i}(H_i)$ by
Lemma~\ref{lem:acl-trans} and is therefore a free factor of $A_i$; write
$A_i=A_i'*C_i$. Then
\[
  F=(A_1'*\dots*A_k')*(C_1*\dots*C_k),
\]
so $F':=A_1'*\dots*A_k'$ is a free factor of $F$ containing $H$, whence
$\acl_F(H)\le F'$. Conversely $A_i'=\acl_F(H_i)\le\acl_F(H)$ by monotonicity
for every $i$, so $F'\le\acl_F(H)$. Hence $\acl_F(H)=F'$; the final assertion
is the case $A_i'=A_i$ for all $i$.
\end{proof}
\begin{definition}\label{def:free-prof-factor}
Let $\Fhat$ be the profinite completion of a finitely generated free group. A
closed subgroup $\widehat L\le\Fhat$ is a \emph{free profinite factor} of
$\Fhat$ if there is a closed subgroup $\widehat M\le\Fhat$ with
$\Fhat=\widehat L\amalg\widehat M$, where $\amalg$ denotes the free product in
the category of profinite groups \cite[Ch.~9]{RZ}. Throughout, $\rk$ denotes the
(topological) rank of a free, or free profinite, group.
\end{definition}
\begin{lemma}\label{lem:smallest-factor}
Let $F$ be a finitely generated free group and $K\le F$ a non-trivial finitely
generated subgroup. If $\widehat L$ is a free profinite factor of $\Fhat$ with
$\overline K\le\widehat L$, then $\overline{\acl_F(K)}\le\widehat L$.
Consequently $\overline{\acl_F(K)}$ is the smallest free profinite factor of
$\Fhat$ containing $\overline K$.
\end{lemma}
\begin{proof}
Write $\Fhat=\widehat L\amalg\widehat M$. The argument of
\cite[Proposition 6.2]{JSZ} applies verbatim after replacing the cyclic
subgroup generated by a word with $K$ itself; we include the short reduction
to make this point explicit.  Apply the relative splitting theorem
\cite[Theorem 4.1]{JSZ} to the subgroup $K\le\widehat L$.  It gives a free
factor $F_1\le F$ containing $K$ such that $\overline{F_1}$ lies in a conjugate
of $\widehat L$ or of $\widehat M$.  Since
$1\neq\overline K\le\overline{F_1}\cap\widehat L$, the
trivial-intersection property for distinct conjugates of free profinite
factors \cite[Proposition 9.1.12]{RZ} forces that conjugate to be
$\widehat L$ itself.  Thus $\overline{F_1}\le\widehat L$.  As
$\acl_F(K)\le F_1$, we obtain
$\overline{\acl_F(K)}\le\overline{F_1}\le\widehat L$.
\end{proof}
\begin{lemma}\label{lem:free-prof-prod}
Let $F=A_1*\dots*A_k$ be an internal free product decomposition of a finitely
generated free group. Then, inside $\Fhat$, the closures satisfy
\[
  \Fhat=\overline{A_1}\amalg\dots\amalg\overline{A_k},
\]
a free profinite product.
\end{lemma}
\begin{proof}
Each $A_i$ is a free factor, hence finitely generated and, by M.~Hall's theorem,
closed in the profinite topology of $F$; its closure in $\Fhat$ is therefore
canonically $\widehat{A_i}$, free profinite of rank $\rk A_i$. Choose a basis
$X_i$ of $A_i$, so that $X=\bigsqcup_iX_i$ is a basis of $F$ and $\Fhat$ is the
free profinite group on $X$. The free profinite product
$\widehat{A_1}\amalg\dots\amalg\widehat{A_k}$ is free profinite on $X$ as well
\cite[Ch.~9]{RZ}, and the canonical continuous homomorphism
$\widehat{A_1}\amalg\dots\amalg\widehat{A_k}\to\Fhat$ carries the basis $X$ to
the basis $X$; hence it is an isomorphism.
\end{proof}
\begin{lemma}\label{lem:measure-stable}
Let $F$ and $P$ be finitely generated free groups and let $K\le F$ be finitely
generated. Then, for every finite group $G$ and every $h\in\Hom(K,G)$,
\[
  \Pr\nolimits^{F}_{K,G}(h)=\Pr\nolimits^{F*P}_{K,G}(h).
\]
\end{lemma}
\begin{proof}
The natural identification
$\Hom(F*P,G)\cong\Hom(F,G)\times\Hom(P,G)$ identifies the uniform measure with
the product measure, while restriction to $K\le F$ depends only on the first
coordinate. Hence
\begin{align*}
  \Pr\nolimits^{F*P}_{K,G}(h)
  &=\frac{|\{\varphi\in\Hom(F*P,G)\mid\varphi|_K=h\}|}{|\Hom(F*P,G)|}\\[4pt]
  &=\frac{|\{\psi\in\Hom(F,G)\mid\psi|_K=h\}|\,|\Hom(P,G)|}
         {|\Hom(F,G)|\,|\Hom(P,G)|}
   =\Pr\nolimits^{F}_{K,G}(h). \qedhere
\end{align*}
\end{proof}

\section{Setwise versus marked profinite rigidity}\label{sec:setwise-vs-marked}

Marked rigidity implies setwise rigidity by definition. In this section we
identify the precise obstruction to the converse, record invariants detected by
the induced measures, and prove stability under free extension.

\subsection{Measure-preserving automorphisms}\label{sec:mp}

Throughout this subsection, $G$ is finitely generated and $H\le G$ is finitely
generated. We first record two elementary facts: ambient
automorphisms induce measure equivalences, and measure equivalences are closed
under composition and inversion.

\begin{lemma}\label{lem:aut}
For every $\psi\in\Aut(G)$ and every $H\le G$, the restriction
$\psi|_H\colon H\to\psi(H)$ is a measure equivalence.
\end{lemma}

\begin{proof}
Fix a finite group $K$.  The map
$\psi^{*}\colon\Hom(G,K)\to\Hom(G,K)$, $\varphi\mapsto\varphi\circ\psi$, is a
bijection, hence preserves the uniform measure.  For $\varphi\in\Hom(G,K)$ put
$\varphi'=\varphi\circ\psi$.  Then
$\varphi'|_{H}=(\varphi|_{\psi(H)})\circ(\psi|_{H})$, so for every
$h\in\Hom(\psi(H),K)$,
\[
  \varphi|_{\psi(H)}=h
  \iff
  \varphi'|_{H}=h\circ\psi|_{H}.
\]
Since $\psi^{*}$ is measure preserving, the two events have equal probability;
that is,
$\Pr_{\psi(H),K}(h)=\Pr_{H,K}(h\circ\psi|_H)$,
which is \eqref{eq:meq} with $L=\psi(H)$ and $\alpha=\psi|_H$.
\end{proof}

\begin{lemma}\label{lem:groupoid}
Measure equivalences are closed under composition and inversion.
\end{lemma}

\begin{proof}
Let $\alpha\colon H\to L$ and $\beta\colon L\to M$ satisfy \eqref{eq:meq}.  For
$h\in\Hom(M,K)$,
\[
  \Pr\nolimits_{M,K}(h)=\Pr\nolimits_{L,K}(h\circ\beta)
   =\Pr\nolimits_{H,K}(h\circ\beta\circ\alpha),
\]
so $\beta\circ\alpha$ is a measure equivalence.  For the inverse, given
$g\in\Hom(H,K)$ apply \eqref{eq:meq} to $h=g\circ\alpha^{-1}$, obtaining
$\Pr_{L,K}(g\circ\alpha^{-1})=\Pr_{H,K}(g)$, which is \eqref{eq:meq} for
$\alpha^{-1}\colon L\to H$.
\end{proof}

By Lemma~\ref{lem:groupoid},
\[
  \MP_G(H):=\{\gamma\in\Aut(H):\gamma\text{ is a measure equivalence }H\to H\}
\]
is a subgroup of $\Aut(H)$; we call it the group of \emph{measure-preserving
automorphisms of $H$ relative to $G$}. Let
\[
  \rho_H\colon\Stab_{\Aut(G)}(H)\longrightarrow\Aut(H),
  \qquad \psi\longmapsto\psi|_H,
\]
be restriction from the setwise stabiliser. Lemma~\ref{lem:aut} gives
$\operatorname{im}\rho_H\le\MP_G(H)$. Theorem~\ref{thm:setwise-marked} says
that equality is precisely the additional condition needed to pass from
setwise to marked rigidity.

\begin{proof}[Proof of Theorem~\ref{thm:setwise-marked}]
$(\Leftarrow)$  Assume (i) and (ii), and let $\alpha\colon H\to L$ be a measure
equivalence with $L\le G$ finitely generated.  By (i), there is
$\psi\in\Aut(G)$ with $\psi(H)=L$.  By Lemma~\ref{lem:aut} the restriction
$\psi|_H\colon H\to L$ is itself a measure equivalence, so by
Lemma~\ref{lem:groupoid}
\[
  \gamma:=(\psi|_H)^{-1}\circ\alpha
\]
is a measure equivalence $H\to H$, that is, $\gamma\in\MP_G(H)$.  By (ii) there
is $\delta\in\Stab_{\Aut(G)}(H)$ with $\delta|_H=\gamma$.  Put
$\varphi:=\psi\circ\delta\in\Aut(G)$.  Then
\[
  \varphi(H)=\psi(\delta(H))=\psi(H)=L,
  \qquad
  \varphi|_H=(\psi|_H)\circ\gamma=\alpha .
\]
Thus $H$ is profinitely rigid in $G$.

$(\Rightarrow)$  Assume $H$ is profinitely rigid in $G$.  Condition (i) is
immediate from the definitions.  For (ii), let $\gamma\in\MP_G(H)$.  Then
$\gamma$ is a measure equivalence $H\to H$, so profinite rigidity supplies
$\varphi\in\Aut(G)$ with $\varphi(H)=H$ and $\varphi|_H=\gamma$.  The first
condition says $\varphi\in\Stab_{\Aut(G)}(H)$ and the second says
$\rho_H(\varphi)=\gamma$; hence $\gamma\in\operatorname{im}\rho_H$.
\end{proof}

When $G$ is LERF---in particular, when $G$ is a finitely generated free
group---Proposition~\ref{prop:lerf} gives the sharper description
\eqref{eq:MP-profinite}.  In that setting condition~(ii) is exactly a
discrete-versus-profinite extension condition.  The abstract
Theorem~\ref{thm:setwise-marked}, however, requires no residual finiteness or
LERF hypothesis.  Condition~(ii) holds, in particular, whenever $\rho_H$ is
surjective.

\begin{proof}[Proof of Corollary~\ref{cor:ff-setwise-marked}]
Write $F=H*H'$.  Given any $\gamma\in\Aut(H)$, the automorphism
$\gamma*\mathrm{id}_{H'}$ of $F$ stabilises $H$ and restricts to $\gamma$ on $H$.
Hence $\rho_H$ is surjective, so
$\operatorname{im}\rho_H=\Aut(H)\supseteq\MP_F(H)$, and condition (ii) of
Theorem~\ref{thm:setwise-marked} holds.  The corollary follows.
\end{proof}

\subsection{Measure invariants}\label{sec:invariants}

We next record several invariants of the measure-equivalence class. Throughout
this subsection, $F$ is free of rank $n$, $H,L\le F$ are nontrivial finitely
generated subgroups, and
$\pi\colon F\to\mathbb Z^n$ is abelianisation, with continuous extension
$\widehat\pi\colon\Fhat\to\Zhat^{\,n}$. The algebraic-closure and abelianisation
invariants below will be used repeatedly in the reduction arguments.

For a subgroup $M\le\mathbb Z^n$ of rank $r$, its \emph{invariant factors}
are the Smith invariants $d_1\mid\cdots\mid d_r$ of the inclusion
$M\hookrightarrow\mathbb Z^n$. Thus, in suitable bases,
$M=\bigoplus_{i=1}^r d_i\mathbb Z e_i$. These invariants record the embedding of
$M$ in $\mathbb Z^n$, not merely its abstract isomorphism type. If $M$ has finite
index, then $r=n$ and all $d_i$ are positive.

\begin{lemma}\label{lem:abelian-invariant}
Let $H,L\le F_n$ be finitely generated subgroups which are measure equivalent via
an isomorphism $\alpha$.  Then the images $\pi(H)$ and $\pi(L)$ in
$\mathbb Z^{n}$ have the same invariant factors; in particular they lie in the
same $GL_n(\mathbb Z)$-orbit.
\end{lemma}

\begin{proof}
By Proposition~\ref{prop:lerf}, there is $\Theta\in\Aut(\widehat{F_n})$ with
$\Theta|_{\overline H}=\widehat\alpha$, so $\Theta(\overline H)=\overline L$.
The kernel of $\widehat\pi$ is the closed derived subgroup of $\widehat{F_n}$,
which by Lemma~\ref{lem:verbal-closure} (with $V=\{[x,y]\}$) equals
$\overline{[F_n,F_n]}$ and is $\Theta$-invariant by
Proposition~\ref{prop:verbal-invariant}.  Hence $\Theta$ descends to
\[
  \Theta^{\mathrm{ab}}\in\Aut_{\mathrm{cont}}(\Zhat^{\,n})=GL_n(\Zhat).
\]
Put $M=\pi(H)$ and $N=\pi(L)$.  As $\overline H$ is compact,
$\widehat\pi(\overline H)$ is closed; it contains $M$ and is contained in
$\overline M$, so $\widehat\pi(\overline H)=\overline M=M\Zhat$, and likewise
$\widehat\pi(\overline L)=N\Zhat$.  Therefore
\[
  \Theta^{\mathrm{ab}}\bigl(M\Zhat\bigr)=N\Zhat .
\]
Let $d_1\mid\dots\mid d_r$ and $e_1\mid\dots\mid e_s$ be the invariant factors of
$M$ and $N$, so that in suitable bases $M\Zhat=\bigoplus_{i\le r}d_i\Zhat$ and
$N\Zhat=\bigoplus_{j\le s}e_j\Zhat$.  Since $\Zhat=\prod_p\mathbb Z_p$ and
$GL_n(\Zhat)=\prod_pGL_n(\mathbb Z_p)$, the displayed equality says that for
every prime $p$ the $\mathbb Z_p$-lattices $M\mathbb Z_p$ and $N\mathbb Z_p$ are
$GL_n(\mathbb Z_p)$-equivalent.  Over the discrete valuation ring $\mathbb Z_p$,
elementary divisors are a complete invariant, so $r=s$ and $v_p(d_i)=v_p(e_i)$
for all $i$ and all $p$; taking the $d_i,e_i$ positive gives $d_i=e_i$.  Two
subgroups of $\mathbb Z^{n}$ with the same invariant factors are
$GL_n(\mathbb Z)$-equivalent, by Smith normal form.
\end{proof}

\begin{proposition}\label{prop:invariants}
Suppose $H$ and $L$ are measure equivalent via $\alpha$.  Then:
\begin{enumerate}
\item[(i)]   $\rk H=\rk L$;
\item[(ii)]  $H$ has finite index in $F$ if and only if $L$ does, and then
             $[F:H]=[F:L]$;
\item[(iii)] $H$ is a free factor of $F$ if and only if $L$ is;
\item[(iv)]  $\rk\acl_F(H)=\rk\acl_F(L)$; consequently there is
             $\beta\in\Aut(F)$ with $\beta(\acl_F(H))=\acl_F(L)$;
\item[(v)]   the images of $H$ and $L$ in $F^{\mathrm{ab}}\cong\mathbb Z^{n}$
             have the same invariant factors; in particular they lie in the same
             $GL_n(\mathbb Z)$-orbit.
\end{enumerate}
\end{proposition}

\begin{proof}
By Proposition~\ref{prop:lerf}, there is $\Theta\in\Aut(\Fhat)$ with
$\Theta|_{\overline H}=\widehat\alpha$; in particular
$\Theta(\overline H)=\overline L$.  We use this $\Theta$ throughout.

(i) is immediate since $\alpha$ is an isomorphism.

(ii)  $H$ and $L$ are finitely generated, hence closed in the profinite topology
by M.~Hall's theorem.  A closed subgroup of $F$ has finite index if and only if
its closure is open in $\Fhat$, and in that case $F/H\to\Fhat/\overline H$ is a
bijection of coset spaces, so $[\Fhat:\overline H]=[F:H]$.  A topological
automorphism carries open subgroups to open subgroups and preserves their
indices; applying this to $\Theta$ gives the claim.

(iii)  For every finite $G$, the map $h\mapsto h\circ\alpha$ is a bijection
$\Hom(L,G)\to\Hom(H,G)$ carrying $\mu_{L,G}$ to $\mu_{H,G}$.  Hence $\mu_{H,G}$
is uniform for all $G$ if and only if $\mu_{L,G}$ is, and
Proposition~\ref{prop:uniform} converts this into the assertion.

(iv)  By Lemma~\ref{lem:smallest-factor}, the subgroup $\overline{\acl_F(K)}$ is
the smallest free profinite factor of $\Fhat$ containing $\overline K$, for every
nontrivial finitely generated $K\le F$.  A continuous automorphism of $\Fhat$
carries free profinite factors to free profinite factors and preserves
inclusions, so
\[
  \Theta\bigl(\overline{\acl_F(H)}\bigr)=\overline{\acl_F(L)} .
\]
Both sides are free profinite of ranks $\rk\acl_F(H)$ and $\rk\acl_F(L)$
respectively, and topological rank is an isomorphism invariant, so these ranks
agree.  Free factors of $F$ of equal rank lie in one $\Aut(F)$-orbit: extend
bases of the two factors to bases of $F$ and map one to the other.

(v) is Lemma~\ref{lem:abelian-invariant}.
\end{proof}
As a first consequence, free factors are rigid in both senses.

\begin{corollary}
Every free factor $H$ of a finitely generated free group $F$ is
setwise profinitely rigid, and profinitely rigid, in $F$.
\end{corollary}

\begin{proof}
Let $L\le F$ be finitely generated and measure equivalent to $H$
via $\alpha$. By Proposition \ref{prop:invariants}(iii), $L$ is a free factor of $F$,
and by (i), it has the same rank as $H$.
Free factors of equal rank lie in a single $\operatorname{Aut}(F)$-orbit:
extend bases of $H$ and $L$ to bases of $F$ and map one extended
basis to the other. Hence $L\in\operatorname{Aut}(F)\cdot H$, so $H$
is setwise profinitely rigid. Since $H$ is a free factor,
Corollary \ref{cor:ff-setwise-marked} upgrades this to profinite rigidity.
\end{proof}
\begin{example}\label{ex:invariants}
The abelianisation invariant in Proposition~\ref{prop:invariants}(v) is
independent of the preceding four invariants.  Take
\[
  H=\langle x_1^{2},\,x_2^{2}\rangle,\qquad L=\langle x_1,\,x_2^{4}\rangle
\]
in $F_2=\langle x_1,x_2\rangle$.  Both are free of rank $2$, both have infinite
index, neither is a free factor, and both have algebraic closure $F_2$; so
(i)--(iv) fail to separate them.  Their abelianised images are
$2\mathbb Z\oplus2\mathbb Z$ and $\mathbb Z\oplus4\mathbb Z$, with invariant
factors $(2,2)$ and $(1,4)$ respectively, so by (v) they are not measure
equivalent.  By contrast $\langle x_1^{2},x_2^{3}\rangle$ and
$\langle x_1^{6},x_2\rangle$ \emph{do} have the same invariant factors $(1,6)$,
by the Chinese remainder theorem, so (v) says nothing about that pair.
\end{example}

\subsection{Universality and the reduction principle}\label{sec:universal}

The marked and setwise statements are proved simultaneously; the only difference is
that in the marked case the witnessing ambient automorphism must realise the prescribed
measure equivalence on $K$.

\begin{proof}[Proof of Theorem~\ref{thm:universal}]
We prove (i) and (ii) simultaneously. Throughout the proof, ``rigid'' means
respectively marked or setwise profinite rigidity. In the marked case we retain
the chosen measure equivalence $\alpha$ and require each witnessing ambient
automorphism to restrict to $\alpha$; in the setwise case this restriction
clause is omitted.

We may assume $K\neq1$, since the trivial subgroup is rigid in both senses in
every ambient group. Put $A=\acl_F(K)$ and write $F=A*C$. Since
$F*P=A*(C*P)$, Lemma~\ref{lem:measure-stable} shows that the measures induced on
$K$ are unchanged by adjoining a free factor. It is therefore enough to prove
the theorem when $K$ is algebraic in $F$, that is, when $\acl_F(K)=F$: the
algebraic case applied to $K\le A$ with free complements $C$ and $C*P$ then
identifies rigidity in $A$, $F$, and $F*P$. In the algebraic case
Lemma~\ref{lem:acl-trans} gives $\acl_{F*P}(K)=F$.

\medskip\noindent\emph{From $F$ to $F*P$.}
Let $L\le F*P$ be finitely generated and choose a measure equivalence
$\alpha\colon K\to L$. By Proposition~\ref{prop:lerf}, there is
$\Phi\in\Aut(\widehat{F*P})$ with $\Phi(\overline K)=\overline L$. Since
$\overline F$ is, by Lemma~\ref{lem:smallest-factor}, the smallest free
profinite factor containing $\overline K$, we have
\[
  \Phi(\overline F)=\overline{\acl_{F*P}(L)}.
\]
Thus $\rk\acl_{F*P}(L)=\rk F$. If
$F*P=\acl_{F*P}(L)*P'$, then $\rk P'=\rk P$; hence there is
$\beta\in\Aut(F*P)$ with $\beta(\acl_{F*P}(L))=F$. Replacing
$(L,\alpha)$ by $(\beta(L),\beta\circ\alpha)$ preserves measure
equivalence by Lemmas~\ref{lem:aut} and~\ref{lem:groupoid}; composing the
final witnessing automorphism with $\beta^{-1}$ recovers the original pair.
Thus we may assume $L\le F$. Lemma~\ref{lem:measure-stable} then identifies the comparison of
$K$ and $L$ in $F*P$ with the same comparison in $F$. Rigidity of $K$ in $F$
therefore gives $\varphi\in\Aut(F)$ with $\varphi(K)=L$ and, in the marked
case, $\varphi|_K=\alpha$. Extending $\varphi$ by the identity on $P$ proves
rigidity of $K$ in $F*P$.

\medskip\noindent\emph{From $F*P$ to $F$.}
Let $L\le F$ be finitely generated and choose a measure equivalence
$\alpha\colon K\to L$ computed in $F$. By Lemma~\ref{lem:measure-stable}, it is
also a measure equivalence in $F*P$. Rigidity there gives
$\varphi\in\Aut(F*P)$ with $\varphi(K)=L$ and, in the marked case,
$\varphi|_K=\alpha$. Automorphisms preserve algebraic closure, so
\[
  \varphi(F)=\varphi(\acl_{F*P}(K))
  =\acl_{F*P}(L)=\acl_F(L)\le F,
\]
where the equality $\acl_{F*P}(L)=\acl_F(L)$ is Lemma~\ref{lem:acl-trans}. Since $\varphi(F)$ is
a free factor of $F$ with the same finite rank as $F$, it follows that
$\varphi(F)=F$. Hence $\varphi|_F\in\Aut(F)$ witnesses rigidity of $K$ in $F$
(and restricts to $\alpha$ in the marked case).

The initial reduction through $A=\acl_F(K)$ also shows that either rigidity
notion for $K\le F$ is equivalent to the corresponding notion in
$\acl_F(K)$.
\end{proof}

We isolate the downward half of Theorem~\ref{thm:universal}, together with two
normalisations that follow from Proposition~\ref{prop:invariants}, as a single
statement.  It is used so often below that giving it a name shortens several
proofs considerably; part (iii) in particular replaces, by one citation, a
reduction step that would otherwise be carried out twice in the proof of
Theorem~\ref{thm:free-prod}.

\begin{proposition}[Reduction to the algebraic closure]\label{prop:reduce}
Let $F$ be a finitely generated free group.
\begin{enumerate}
\item[(i)] Let $H\le F$ be nontrivial and finitely generated and put
  $A=\acl_F(H)$.  Then $H$ is profinitely rigid in $F$ if and only if it is
  profinitely rigid in $A$, and likewise for setwise profinite rigidity.
\item[(ii)] In testing Definition~\ref{def:rigid} or Definition~\ref{def:setwise}
  for such an $H$, one may always assume that the comparison subgroup $L$
  satisfies $L\le A$ and $\acl_F(L)=A$.
\item[(iii)] Let $F=A_1*\dots*A_k$ be an internal free product decomposition, let
  $H_i\le A_i$ be finitely generated, and put $H=H_1*\dots*H_k$ and
  $A_i'=\acl_F(H_i)$.  Then $F':=A_1'*\dots*A_k'$ is a free factor of $F$
  containing $H$, one has $\acl_F(H)=F'$, and each of the statements
  ``$H_i$ is (setwise) profinitely rigid'' and ``$H$ is (setwise) profinitely
  rigid'' holds in $F$ if and only if it holds in $F'$.  Consequently, in proving
  such a statement one may assume $A_i=\acl_F(H_i)$ for every $i$.
\end{enumerate}
\end{proposition}

\begin{proof}
(i)  $A$ is a free factor, say $F=A*C$, so this is Theorem~\ref{thm:universal}
applied with the roles of $F$ and $F*P$ played by $A$ and $A*C$.

(ii)  Let $L\le F$ be measure equivalent to $H$ via $\alpha$.  By
Proposition~\ref{prop:invariants}(iv) there is $\beta\in\Aut(F)$ with
$\beta(\acl_F(L))=A$.  Replacing $L$ by $\beta(L)$ and $\alpha$ by
$\beta\circ\alpha$ changes neither the hypothesis nor the conclusion of either
definition---for the hypothesis this is Lemmas~\ref{lem:aut} and
\ref{lem:groupoid}, and for the conclusion it is because $\Aut(F)$ acts---and
now $L\le\acl_F(L)=A$ with $\acl_F(L)=A$.  Finally, by
Lemma~\ref{lem:measure-stable} the distributions induced by $H$ and $L$ are the
same whether computed in $F$ or in $A$.

(iii)  By Lemma~\ref{lem:acl-trans}, $A_i'=\acl_{A_i}(H_i)$ is a free factor of
$A_i$; write $A_i=A_i'*C_i$.  Then
\[
  F=A_1'*C_1*\dots*A_k'*C_k=(A_1'*\dots*A_k')*(C_1*\dots*C_k),
\]
so $F'=A_1'*\dots*A_k'$ is a free factor of $F$, and it contains
$H=H_1*\dots*H_k$ since $H_i\le A_i'$.  That $\acl_F(H)=F'$ is
Lemma~\ref{lem:acl-freeprod} together with Lemma~\ref{lem:acl-trans}.  Since
$A_i'$ is a free factor of $F'$ and $F'$ is a free factor of $F$, part~(i)
applied to $H_i$ (whose algebraic closure in $F$ and in $F'$ is $A_i'$ in both
cases, by Lemma~\ref{lem:acl-trans}) shows that $H_i$ is (setwise) profinitely
rigid in $F$ if and only if it is (setwise) profinitely rigid in $F'$; and the
same argument applied to $H$, whose algebraic closure is $F'$, gives the
corresponding statement for $H$.  Replacing $F$ by $F'$ and $A_i$ by $A_i'$ is
therefore harmless, and after that replacement $A_i=\acl_F(H_i)$.
\end{proof}

Theorem~\ref{thm:universal} moves the ambient group in two directions: upwards
along free extensions $F\rightsquigarrow F*P$, and downwards to $\acl_F(K)$. 

\section{Setwise profinite rigidity}\label{sec:setwise-rigid}

Setwise rigidity has a source of examples that does not directly address the
marked problem: if the closure of a finitely generated subgroup is fixed by
$\Aut(\Fhat)$, then its measure-equivalence class collapses to a single
subgroup. We first develop this observation for finite-index verbal subgroups
and then classify the measure classes of all subgroups containing
$[F_n,F_n]$.

\subsection{Finite-index verbal subgroups}\label{sec:verbal}

We begin with subgroups whose closures are invariant under every automorphism of
$\widehat{F_n}$. For such subgroups, setwise rigidity strengthens to uniqueness
inside the measure-equivalence class.

\begin{definition}\label{def:profchar}
A finite-index subgroup $H\le F_n$ is \emph{profinitely characteristic} if
$\Theta(\overline H)=\overline H$ for every continuous automorphism
$\Theta\in\Aut(\widehat{F_n})$, where $\overline H$ denotes the closure of $H$ in
$\widehat{F_n}$.
\end{definition}

For nontrivial finitely generated subgroups, the finite-index condition is in
fact forced by profinite invariance. We use the following standard fact.

\begin{lemma}\label{lem:normal-fi}
Let $n\ge1$ and let $H\trianglelefteq F_n$ be a non-trivial finitely generated
normal subgroup.  Then $[F_n:H]<\infty$.
\end{lemma}

\begin{proof}
Realise $F_n$ as $\pi_1$ of a wedge $X$ of $n$ circles and let $Y\to X$ be the
covering with $\pi_1(Y)=H$.  Let $C\subseteq Y$ be the core of $Y$, that is, the
union of the images of all reduced circuits in $Y$; it is non-empty because
$H\neq1$, and finite because $H$ is finitely generated.  The core is defined
without reference to a basepoint, so it is preserved by every automorphism of the
graph $Y$, in particular by every deck transformation.  As $H$ is normal the
covering is regular, so the deck group $F_n/H$ acts freely on the vertices of
$Y$, hence freely on the finitely many vertices of $C$.  Therefore
$|F_n/H|\le|V(C)|<\infty$.
\end{proof}

The covering-space proof avoids any appeal to the Nielsen--Schreier index
formula, whose finite-index hypothesis would be circular at this point.

\begin{proposition}\label{prop:fi-automatic}
Let $H\le F_n$ be non-trivial and finitely generated, and suppose
$\Theta(\overline H)=\overline H$ for every $\Theta\in\Aut(\widehat{F_n})$.  Then
$H$ is characteristic in $F_n$ and $[F_n:H]<\infty$; that is, $H$ is a profinitely
characteristic subgroup in the sense of Definition~\ref{def:profchar}, the
finite-index clause being automatic.
\end{proposition}

\begin{proof}
Let $\varphi\in\Aut(F_n)$ and let $\widehat\varphi\in\Aut(\widehat{F_n})$ be its
continuous extension.  Then $\widehat\varphi(\overline H)$ is closed and contains
$\varphi(H)$, so $\overline{\varphi(H)}\le\widehat\varphi(\overline H)$; applying
the same to $\varphi^{-1}$ and $\varphi(H)$ gives the reverse inclusion, whence
$\widehat\varphi(\overline H)=\overline{\varphi(H)}$.  By hypothesis
$\overline{\varphi(H)}=\overline H$.  Both $H$ and $\varphi(H)$ are finitely
generated, hence closed in the profinite topology by M.~Hall's theorem
\cite{Hall}, so
\[
  \varphi(H)=\overline{\varphi(H)}\cap F_n=\overline H\cap F_n=H .
\]
Thus $H$ is characteristic, in particular normal, and
Lemma~\ref{lem:normal-fi} gives $[F_n:H]<\infty$.
\end{proof}

\begin{remark}\label{rmk:fi-automatic}
The finite-index condition in Definition~\ref{def:profchar} is automatic once
$H$ is nontrivial and finitely generated. Both hypotheses are necessary:
$\overline{[F_n,F_n]}$ is $\Aut(\widehat{F_n})$-invariant while
$[F_n,F_n]$ has infinite index for $n\ge2$, and the trivial subgroup is also
invariant.
\end{remark}

Profinite characteristicity is stronger than ordinary characteristicity and is
exactly the invariance needed below. Proposition~\ref{prop:verbal-invariant}
shows, in particular, that every finite-index verbal subgroup is profinitely
characteristic.

\begin{example}\label{ex:verbal}
The kernel of the mod-$m$ abelianisation,
\[
  F_n^{\,m}[F_n,F_n]=\ker\bigl(F_n\twoheadrightarrow(\mathbb Z/m\mathbb Z)^{n}\bigr),
\]
is verbal, with defining words $x^{m}$ and $[x,y]$, and has index $m^{n}$; for
$m=p$ prime it is the kernel of the mod-$p$ abelianisation.  For odd $p$ the
subgroup $F_n^{\,p}\gamma_3(F_n)$, whose quotient is the free class-two exponent-$p$
group, is a non-abelian nilpotent example of index $p^{\,n+\binom n2}$.

More generally, for any finite group $G$ the subgroup
\[
  K_{F_n}(G)=\bigcap_{\varphi\in\Epi(F_n,G)}\ker\varphi
  =\bigcap\{N\fii F_n : F_n/N\cong G\}
\]
of Section~\ref{sec:tuples} is profinitely characteristic.  The argument is
carried out in $\widehat{F_n}$, as Definition~\ref{def:profchar} requires.  Since
$F_n$ is finitely generated and $G$ is finite, $\Epi(F_n,G)$ is finite, so the
intersection above is a \emph{finite} intersection of finite-index normal
subgroups; this matters, because taking closures commutes with finite
intersections of finite-index subgroups but not with arbitrary ones.  Under
the bijection $N\mapsto\overline N$ between the finite-index normal subgroups of
$F_n$ and the open normal subgroups of $\widehat{F_n}$ one has
$\widehat{F_n}/\overline N\cong F_n/N$; hence
\[
  \overline{K_{F_n}(G)}=\bigcap\{U\trianglelefteq_{o}\widehat{F_n}
    :\widehat{F_n}/U\cong G\}.
\]
A continuous automorphism of $\widehat{F_n}$ permutes the open normal subgroups
with quotient isomorphic to $G$, hence fixes this intersection.  The finite-index
subgroup $F_n(d)=\bigcap_{[F_n:K]\le d}K$, the intersection of all subgroups of
index at most $d$, is profinitely characteristic for the same reason.  It is the
passage to the intersection that is essential: for $n\ge2$ the
group $F_n$ has $2^{n}-1$ subgroups of index two, which $\Aut(F_n)$ permutes, so
no single one is characteristic, yet their intersection
$F_n^{\,2}[F_n,F_n]=K_{F_n}(\mathbb Z/2\mathbb Z)$ is.
\end{example}

\begin{theorem}\label{thm:measure-detected}
Let $H\le F_n$ be a profinitely characteristic finite-index subgroup, for
instance, any finite-index verbal subgroup.  If $L\le F_n$ is finitely generated
and measure equivalent to $H$, then $L=H$.  In particular $H$ is setwise
profinitely rigid in $F_n$.
\end{theorem}

\begin{proof}
Let $\alpha\colon H\to L$ be a measure equivalence.  By
Proposition~\ref{prop:lerf} there is a continuous automorphism
$\widetilde\alpha\in\Aut(\widehat{F_n})$ with
$\widetilde\alpha(\overline H)=\overline L$; this existence is the substantive
input, resting on the Garrido--Jaikin-Zapirain criterion of
Lemma~\ref{lem:gj}.  Because $H$ is profinitely characteristic,
$\widetilde\alpha(\overline H)=\overline H$, whence
$\overline L=\widetilde\alpha(\overline H)=\overline H$.  Both $H$ and $L$ are
finitely generated, hence closed in the profinite topology by M.~Hall's theorem
\cite{Hall}; thus $H=\overline H\cap F_n$ and $L=\overline L\cap F_n$, and
therefore
\[
  L=\overline L\cap F_n=\overline H\cap F_n=H. \qedhere
\]
\end{proof}

Thus a profinitely characteristic subgroup is \emph{measure-detected}: every
finitely generated subgroup measure equivalent to it is equal to it. This is
strictly stronger than setwise rigidity.

\subsection{Subgroups containing the commutator subgroup}\label{sec:commutator}

Profinite characteristicity is restrictive: for $n\ge2$, none of the
$2^n-1$ subgroups of index two in $F_n$ is characteristic. Nevertheless, all
subgroups containing $[F_n,F_n]$ are setwise profinitely rigid. Their measure
classes admit a complete description in terms of Smith normal form.

\begin{theorem}\label{thm:above-commutator}
Let $H\le F_n$ be a nontrivial finitely generated subgroup with
$[F_n,F_n]\le H$.  If $L\le F_n$ is finitely generated and measure equivalent to $H$, then
there is $\varphi\in\Aut(F_n)$ with $\varphi(H)=L$.  Thus $H$ is setwise
profinitely rigid in $F_n$.  Moreover the $\Aut(F_n)$-orbit of $H$, and hence its
measure class, is determined exactly by the invariant factors of
$\pi(H)$ in $\mathbb Z^{n}$.
\end{theorem}

\begin{proof}
First note that the stated equivalence holds.  A subgroup containing $[F_n,F_n]$
is normal, since $F_n/[F_n,F_n]$ is abelian, and the quotient $F_n/H$ is then
abelian; that $H$ has finite index is Lemma~\ref{lem:normal-fi}.  Conversely a
subgroup with $F_n/H$ finite
abelian contains $[F_n,F_n]$, is finitely generated by Nielsen--Schreier, and is
nontrivial because $F_n$ is infinite.  (For $n\ge2$ the word ``nontrivial'' is
automatic, as $[F_n,F_n]\neq1$; it is needed only to exclude $H=1$ in $F_1$.)

Let $\alpha\colon H\to L$ be a measure equivalence and let
$\Theta\in\Aut(\widehat{F_n})$ be as in Proposition~\ref{prop:lerf}, so
$\Theta(\overline H)=\overline L$.  By Lemma~\ref{lem:verbal-closure} applied to
$V=\{[x,y]\}$, the closure $\overline{[F_n,F_n]}$ is the closed derived subgroup
of $\widehat{F_n}$, which is $\Theta$-invariant by
Proposition~\ref{prop:verbal-invariant}.  Since
$\overline H\supseteq\overline{[F_n,F_n]}$ we get
$\overline L=\Theta(\overline H)\supseteq\overline{[F_n,F_n]}$.  The subgroup
$[F_n,F_n]$ is closed in the profinite topology of $F_n$, because
$F_n/[F_n,F_n]\cong\mathbb Z^{n}$ is residually finite; and $L$ is closed by
M.~Hall's theorem.  Hence
\[
  L=\overline L\cap F_n\supseteq\overline{[F_n,F_n]}\cap F_n=[F_n,F_n].
\]
Consequently, $H=\pi^{-1}(M)$ and $L=\pi^{-1}(N)$, where $M=\pi(H)$ and
$N=\pi(L)$.

By Lemma~\ref{lem:abelian-invariant}, the subgroups $M,N\le\mathbb Z^{n}$ have the
same invariant factors, hence there is $\gamma\in GL_n(\mathbb Z)$ with
$\gamma(M)=N$.  The natural map $\Aut(F_n)\to GL_n(\mathbb Z)$ is surjective
(Nielsen), so we may choose $\varphi\in\Aut(F_n)$ with
$\varphi^{\mathrm{ab}}=\gamma$.  Any automorphism of $F_n$ preserves
$[F_n,F_n]$, so
\[
  \varphi(H)=\varphi\bigl(\pi^{-1}(M)\bigr)=\pi^{-1}\bigl(\gamma(M)\bigr)
  =\pi^{-1}(N)=L .
\]

For the last assertion, subgroups of $\mathbb Z^{n}$ with the same invariant
factors lie in one $GL_n(\mathbb Z)$-orbit and conversely, and taking
$\pi$-preimages is an $\Aut(F_n)$-equivariant bijection between the subgroups of
$\mathbb Z^{n}$ and the subgroups of $F_n$ containing $[F_n,F_n]$.
\end{proof}

\begin{corollary}\label{cor:count}
For $n\ge1$ the map
\[
  H\longmapsto(d_1\mid d_2\mid\dots\mid d_n)
\]
sending a finite-index subgroup $H\le F_n$ containing $[F_n,F_n]$ to the
invariant factors of $\pi(H)$---here $\pi(H)$ has finite index in
$\mathbb Z^{n}$, so its rank is $n$ and the list has exactly $n$ positive
terms---is a complete invariant of the measure class of
$H$: two such subgroups are measure equivalent if and only if they have the same
invariant factors, if and only if they lie in the same $\Aut(F_n)$-orbit.  In
particular, there are exactly as many measure classes of index-$m$ subgroups of
$F_n$ containing $[F_n,F_n]$ as there are abelian groups of order $m$ generated
by at most $n$ elements.
\end{corollary}

\begin{proof}
One direction is Theorem~\ref{thm:above-commutator}; the other holds because
subgroups in a common $\Aut(F_n)$-orbit induce the same measures, by
Lemma~\ref{lem:aut}.  The count follows since
$F_n/H\cong\mathbb Z^{n}/\pi(H)\cong\bigoplus_i\mathbb Z/d_i$.
\end{proof}

\begin{remark}
Theorem~\ref{thm:above-commutator} is not a consequence of
Theorem~\ref{thm:measure-detected}.  For example, for $n\ge2$ the subgroup
$H=\pi^{-1}(2\mathbb Z\oplus\mathbb Z^{n-1})$ has index two in $F_n$ and is not
characteristic, so it is not profinitely characteristic and
Theorem~\ref{thm:measure-detected} says nothing about it;
Theorem~\ref{thm:above-commutator} shows that its measure class is its
$\Aut(F_n)$-orbit, which consists of the $2^{n}-1$ subgroups of index two.
Conversely Theorem~\ref{thm:measure-detected} gives a strictly stronger
conclusion for the subgroups it covers, namely $L=H$ on the nose.  The two
results agree, as they must, on $H=\pi^{-1}(m\mathbb Z^{n})=F_n^{\,m}[F_n,F_n]$,
whose $\Aut(F_n)$-orbit is a single subgroup.
\end{remark}

Within this lattice one can also characterise exactly the measure-detected
subgroups: they are precisely the mod-$m$ congruence subgroups.

\begin{theorem}\label{thm:detected-char}
Let $n\ge1$ and let $H\le F_n$ be a subgroup of finite index with
$[F_n,F_n]\le H$.  The following are equivalent.
\begin{enumerate}
\item[(i)]   $H$ is measure-detected: every finitely generated $L\le F_n$
             measure equivalent to $H$ satisfies $L=H$.
\item[(ii)]  $H$ is profinitely characteristic.
\item[(iii)] $H$ is characteristic in $F_n$.
\item[(iv)]  $\pi(H)\le\mathbb Z^{n}$ is invariant under $GL_n(\mathbb Z)$.
\item[(v)]   $H=F_n^{\,m}[F_n,F_n]$ for some integer $m\ge1$; equivalently, the
             invariant factors of $\pi(H)$ satisfy $d_1=\dots=d_n$.
\end{enumerate}
\end{theorem}

\begin{proof}
We prove
$(\mathrm i)\Rightarrow(\mathrm{iii})\Rightarrow(\mathrm{iv})\Rightarrow(\mathrm v)
\Rightarrow(\mathrm{ii})\Rightarrow(\mathrm i)$.  Throughout we use that
$[F_n,F_n]\le H$ forces $H=\pi^{-1}(\pi(H))$.

$(\mathrm i)\Rightarrow(\mathrm{iii})$.  Let $\varphi\in\Aut(F_n)$ and put
$L=\varphi(H)$.  By Lemma~\ref{lem:aut}, $\varphi|_H\colon H\to L$ is a measure
equivalence; and $L$ is finitely generated by Nielsen--Schreier.  By (i), $L=H$.

$(\mathrm{iii})\Rightarrow(\mathrm{iv})$.  Let $\gamma\in GL_n(\mathbb Z)$.  The
natural map $\Aut(F_n)\to GL_n(\mathbb Z)$ is surjective (Nielsen), so
$\gamma=\varphi^{\mathrm{ab}}$ for some $\varphi\in\Aut(F_n)$, and then
$\gamma(\pi(H))=\pi(\varphi(H))=\pi(H)$.

$(\mathrm{iv})\Rightarrow(\mathrm v)$.  Put $M=\pi(H)$, of finite index in
$\mathbb Z^{n}$, with invariant factors $d_1\mid\dots\mid d_n$, all positive.
Choose a basis $f_1,\dots,f_n$ of $\mathbb Z^{n}$ with
$M=\bigoplus_{i}d_i\mathbb Z f_i$.  For $n=1$, there is nothing to prove.  For
$n\ge2$, let $\gamma\in GL_n(\mathbb Z)$ be the permutation matrix interchanging
$f_1$ and $f_n$.  Then $d_1f_n=\gamma(d_1f_1)\in\gamma(M)=M$, while
$M\cap\mathbb Z f_n=d_n\mathbb Z f_n$; hence $d_n\mid d_1$.  Together with
$d_1\mid d_n$ this gives $d_1=d_n$, so all the $d_i$ are equal to a common value
$m\ge1$ and $M=m\mathbb Z^{n}$.  Therefore
$H=\pi^{-1}(m\mathbb Z^{n})=F_n^{\,m}[F_n,F_n]$.

$(\mathrm v)\Rightarrow(\mathrm{ii})$.  $F_n^{\,m}[F_n,F_n]=V(F_n)$ is the verbal
subgroup for $V=\{x^{m},[x,y]\}$ and has index $m^{n}$, so
Proposition~\ref{prop:verbal-invariant} applies.

$(\mathrm{ii})\Rightarrow(\mathrm i)$ is Theorem~\ref{thm:measure-detected}.
\end{proof}

\begin{corollary}\label{cor:classsize}
Let $H\le F_n$ be of finite index with $[F_n,F_n]\le H$, and let
$d_1\mid\dots\mid d_n$ be the invariant factors of $\pi(H)$.  Then the measure
class of $H$ is in bijection with the set of subgroups of $\mathbb Z^{n}$ having
invariant factors $d_1\mid\dots\mid d_n$.  It is a single subgroup if and only if
$d_1=\dots=d_n$.  For $[F_n:H]=p$ prime it has exactly
\[
  \frac{p^{n}-1}{p-1}
\]
members.
\end{corollary}

\begin{proof}
By Theorem~\ref{thm:above-commutator} the measure class of $H$ is its
$\Aut(F_n)$-orbit, and $\pi$ induces an $\Aut(F_n)$-equivariant bijection between
the subgroups of $F_n$ containing $[F_n,F_n]$ and the subgroups of
$\mathbb Z^{n}$, under which the orbit of $H$ corresponds to the
$GL_n(\mathbb Z)$-orbit of $\pi(H)$; by Smith normal form the latter is the set
of subgroups with the given invariant factors.  The singleton criterion is
Theorem~\ref{thm:detected-char}, $(\mathrm i)\Leftrightarrow(\mathrm v)$.  For
$[F_n:H]=p$ prime the invariant factors are $(1,\dots,1,p)$ and the subgroups in
question are the index-$p$ subgroups of $\mathbb Z^{n}$, that is, the kernels of
the surjections $\mathbb Z^{n}\to\mathbb Z/p\mathbb Z$; two such surjections have
the same kernel exactly when they differ by a scalar in
$(\mathbb Z/p\mathbb Z)^{\times}$, giving
$(p^{n}-1)/(p-1)$ subgroups.
\end{proof}

For $p=2$, this recovers the count $2^n-1$.

\begin{remark}
The same argument applies more generally to a finite-index verbal subgroup
$V(F_n)$ provided every automorphism of $F_n/V(F_n)$ lifts to $F_n$. Under this
hypothesis, every subgroup containing $V(F_n)$ is setwise profinitely rigid.
Indeed, Lemma~\ref{lem:verbal-closure} and
Proposition~\ref{prop:verbal-invariant} make the closed subgroup
$\overline{V(F_n)}=V(\widehat{F_n})$ characteristic in $\widehat{F_n}$, while
finite index gives $V(F_n)=\overline{V(F_n)}\cap F_n$. The lifting hypothesis is
the additional ingredient needed to pass from the quotient back to an ambient
automorphism of $F_n$.
\end{remark}

\section{Marked profinite rigidity and free products}\label{sec:freeproducts}

We now prove the free-product theorem for marked rigidity. The argument has two
parts. First, a measure equivalence on
$H=H_1*\cdots*H_k$ restricts to measure equivalences on the factors $H_i$.
Rigidity of each $H_i$ then produces factorwise ambient automorphisms. Second,
we show that the corresponding images of the ambient factors still generate
$F$, so the factorwise maps glue to a single automorphism.

Throughout the section, fix an internal decomposition
\[
  F=A_1*\cdots*A_k,
  \qquad H_i\le A_i\ \text{finitely generated},
  \qquad H=H_1*\cdots*H_k.
\]

\subsection{A gluing criterion}\label{sec:gluing}

\begin{definition}[Setup]\label{def:setup}
Let $L\le F$ be a finitely generated subgroup and let $\alpha\colon H\to L$ be a
measure equivalence.  For each $i$, set
\[
  L_i:=\alpha(H_i),\qquad \alpha_i:=\alpha|_{H_i}\colon H_i\to L_i .
\]
Since $\alpha$ is an isomorphism and $H=H_1*\dots*H_k$, we obtain an internal
free product decomposition $L=L_1*\dots*L_k$.  Note that the subgroups $L_i$ need
not lie in distinct free factors of $F$.
\end{definition}

The key elementary input is a product formula. For probability measures
$\nu_i$ on finite sets $X_i$, write $\bigotimes_i\nu_i$ for their product
measure on $\prod_iX_i$.

\begin{lemma}\label{lem:product-measure}
Let $F=A_1*\dots*A_k$ be an internal free product decomposition of a finitely
generated free group and let $P_i\le A_i$ be finitely generated subgroups.  Put
$P=P_1*\dots*P_k$.  Then for every finite group $G$, under the canonical
identification $\Hom(P,G)\cong\prod_{i=1}^{k}\Hom(P_i,G)$,
\[
  \mu_{P,G}=\bigotimes_{i=1}^{k}\mu_{P_i,G},
\]
where each $\mu_{P_i,G}$ may equivalently be computed in $A_i$ or in $F$.
\end{lemma}

\begin{proof}
Fix a finite group $G$.  The universal property of the free product gives
$\Hom(F,G)\cong\prod_{i=1}^{k}\Hom(A_i,G)$, under which the uniform measure on
$\Hom(F,G)$ is the product of the uniform measures on the factors.  As
$P_i\le A_i$, the restriction $\varphi\mapsto\varphi|_{P_i}$ depends only on the
$A_i$-coordinate; hence the coordinate restrictions are independent and the joint
law of $(\varphi|_{P_1},\dots,\varphi|_{P_k})$ is the product of the laws of the
$\varphi|_{P_i}$.  That $\mu_{P_i,G}$ is the same computed in $A_i$ or in $F$ is
Lemma~\ref{lem:measure-stable}, applied to the decomposition
$F=A_i*(\ast_{j\neq i}A_j)$.
\end{proof}

\begin{lemma}\label{lem:factorwise}
For each $i$, the isomorphism $\alpha_i\colon H_i\to L_i$ is a measure
equivalence: for every finite group $G$ and every $h_i\in\Hom(L_i,G)$,
\[
  \Pr\nolimits_{L_i,G}(h_i)=\Pr\nolimits_{H_i,G}(h_i\circ\alpha_i).
\]
\end{lemma}

\begin{proof}
Fix a finite group $G$.  By Lemma~\ref{lem:product-measure} applied to
$P_i=H_i$,
\[
  \mu_{H,G}=\bigotimes_{i=1}^{k}\mu_{H_i,G}.
\]
Using $\Hom(L,G)\cong\prod_i\Hom(L_i,G)$, the hypothesis
$\Pr_{L,G}(h)=\Pr_{H,G}(h\circ\alpha)$ becomes
\[
  \Pr\nolimits_{L,G}(h)=\prod_{i=1}^{k}
  \Pr\nolimits_{H_i,G}\bigl(h|_{L_i}\circ\alpha_i\bigr).
\]
Thus $\mu_{L,G}$ is a product measure whose $i$-th marginal is
$h_i\mapsto\Pr_{H_i,G}(h_i\circ\alpha_i)$.  On the other hand, the $i$-th
marginal of $\mu_{L,G}$ is, by definition, the distribution induced by $L_i$,
since restricting a homomorphism to $L$ and then to $L_i$ is the same as
restricting it to $L_i$.  Comparing marginals gives
$\Pr_{L_i,G}(h_i)=\Pr_{H_i,G}(h_i\circ\alpha_i)$.
\end{proof}

\begin{corollary}\label{cor:betai}
Suppose that $H_i$ is profinitely rigid in $F$. Then there exists
$\beta_i\in\Aut(F)$ such that
\[
  \beta_i(H_i)=L_i,\qquad \beta_i|_{H_i}=\alpha_i.
\]
\end{corollary}

\begin{proof}
By Lemma~\ref{lem:factorwise}, $\alpha_i$ is a measure equivalence.  Since
$H_i$ is profinitely rigid in $F$, Definition~\ref{def:rigid} provides such a
$\beta_i$.
\end{proof}

To pass from the factorwise automorphisms to one automorphism of $F$, we use the
following gluing criterion.

\begin{lemma}[Gluing]\label{lem:gluing}
Let $F=A_1*\dots*A_k$ and let $\beta_i\in\Aut(F)$ for each $i$.  Put
$B_i=\beta_i(A_i)$.  If $\langle B_1,\dots,B_k\rangle=F$, then there is a unique
automorphism $\varphi\in\Aut(F)$ with $\varphi|_{A_i}=\beta_i|_{A_i}$ for every
$i$.  In particular $\varphi|_{H_i}=\beta_i|_{H_i}$ for every subgroup
$H_i\le A_i$.
\end{lemma}

\begin{proof}
By the universal property of the free product, the homomorphisms
$A_i\xrightarrow{\beta_i}B_i\hookrightarrow F$ determine a unique endomorphism
$\varphi\colon F\to F$ with $\varphi|_{A_i}=\beta_i|_{A_i}$.  Its image is
$\varphi(F)=\langle B_1,\dots,B_k\rangle=F$, so $\varphi$ is surjective.
Finitely generated free groups are Hopfian, so $\varphi$ is an automorphism.  The
final assertion is immediate from $H_i\le A_i$.
\end{proof}

\subsection{The free-product theorem}\label{sec:general-thm}

The gluing condition of Lemma~\ref{lem:gluing} is in fact automatic.  The key
point is that the profinite automorphism supplied by measure equivalence carries
the smallest free profinite factor containing each $H_i$ to the corresponding
factor for its image.  We first make the reduction of
Proposition~\ref{prop:reduce}(iii), and hence assume throughout that
\begin{equation}\label{eq:normalised}
  F=A_1*\dots*A_k,\qquad A_i=\acl_F(H_i)\quad(1\le i\le k).
\end{equation}

\begin{proof}[Proof of each $H_i$ rigid $\Rightarrow$ $H$ rigid]
Let
\[
  I:=\{i\in\{1,\dots,k\}:H_i\neq1\},
  \qquad
  F_0:=\mathop{\ast}_{i\in I}A_i,
  \qquad
  D:=\mathop{\ast}_{i\notin I}A_i.
\]
Then
\[
  F=F_0*D,
  \qquad
  H=\mathop{\ast}_{i\in I}H_i\le F_0.
\]
If $I=\varnothing$, then $H=1$, which is profinitely rigid in every
ambient group, so there is nothing to prove. Assume therefore that
$I\neq\varnothing$.

Since $F_0$ is a free factor of $F$, Theorem~\ref{thm:universal} shows,
for every $i\in I$, that $H_i$ is profinitely rigid in $F_0$ if and only
if it is profinitely rigid in $F$. The same theorem shows that $H$ is
profinitely rigid in $F_0$ if and only if it is profinitely rigid in
$F$. We may therefore replace $F$ by $F_0$ and delete the factors with
$H_i=1$. Thus, from now on, we assume that $H_i\neq1$ for every $i$.

Normalise as in \eqref{eq:normalised}.

\medskip\noindent\emph{Step 1: factorwise correction.}
Let $L\le F$ be finitely generated and let $\alpha\colon H\to L$ be a measure
equivalence.  Set $L_i=\alpha(H_i)$ and $\alpha_i=\alpha|_{H_i}$ as in
Definition~\ref{def:setup}.  By Lemma~\ref{lem:factorwise}, each
$\alpha_i\colon H_i\to L_i$ is a measure equivalence, and by
Corollary~\ref{cor:betai}, there is $\beta_i\in\Aut(F)$ with
\[
  \beta_i(H_i)=L_i,\qquad \beta_i|_{H_i}=\alpha_i .
\]
Since automorphisms of $F$ preserve algebraic closures,
\begin{equation}\label{eq:star}
  \beta_i(A_i)=\beta_i\bigl(\acl_F(H_i)\bigr)
  =\acl_F\bigl(\beta_i(H_i)\bigr)=\acl_F(L_i)=:B_i .
\end{equation}
In particular each $B_i$ is a free factor of $F$, independent of the choice of
$\beta_i$.

\medskip\noindent\emph{Step 2: the corrected factors generate.}
By Proposition~\ref{prop:lerf} there is $\Theta\in\Aut(\Fhat)$ with
$\Theta|_{\overline H}=\widehat\alpha$.  As $\widehat\alpha$ is a homeomorphism
$\overline H\to\overline L$ carrying $\overline{H_i}$ onto $\overline{L_i}$, we
get $\Theta(\overline{H_i})=\overline{L_i}$ for every $i$.  Now $H_i\neq1$, so
Lemma~\ref{lem:smallest-factor} identifies
$\overline{A_i}=\overline{\acl_F(H_i)}$ as the smallest free profinite factor of
$\Fhat$ containing $\overline{H_i}$, and likewise $\overline{B_i}$ as the
smallest free profinite factor containing $\overline{L_i}$.  A continuous
automorphism preserves the class of free profinite factors and preserves
inclusion, hence carries smallest such factors to smallest such factors:
\[
  \Theta\bigl(\overline{A_i}\bigr)=\overline{B_i}\qquad(1\le i\le k).
\]
By Lemma~\ref{lem:free-prof-prod},
$\Fhat=\overline{A_1}\amalg\dots\amalg\overline{A_k}$; applying the automorphism
$\Theta$ to this decomposition gives
\[
  \Fhat=\Theta\bigl(\overline{A_1}\bigr)\amalg\dots\amalg
        \Theta\bigl(\overline{A_k}\bigr)
       =\overline{B_1}\amalg\dots\amalg\overline{B_k}.
\]
In particular $\overline{B_1},\dots,\overline{B_k}$ topologically generate
$\Fhat$.  Put $M=\langle B_1,\dots,B_k\rangle\le F$.  Each $B_i$ is a free factor
of $F$, hence finitely generated, so $M$ is finitely generated; and the closure of
$M$ in $\Fhat$ is a closed subgroup containing every $\overline{B_i}$, hence
equals $\Fhat$.  Thus $M$ is dense.  Because $F$ is LERF and $M$ is finitely
generated, $M$ is closed in the profinite topology of $F$; a proper closed
subgroup cannot be dense, so
\[
  M=\langle B_1,\dots,B_k\rangle=F .
\]

\medskip\noindent\emph{Step 3: gluing.}
By \eqref{eq:star} and Step 2 the automorphisms $\beta_1,\dots,\beta_k$ satisfy
the hypothesis of Lemma~\ref{lem:gluing}, which produces $\varphi\in\Aut(F)$ with
$\varphi|_{A_i}=\beta_i|_{A_i}$ for every $i$.  Consequently
$\varphi|_{H_i}=\beta_i|_{H_i}=\alpha_i$ and $\varphi(H_i)=L_i$ for every $i$.
Since $H=H_1*\dots*H_k$ and $L=L_1*\dots*L_k$, the universal property of the free
product gives $\varphi|_H=\alpha$ and $\varphi(H)=L$.  Hence $H$ is profinitely
rigid in $F$.

The final assertion follows from Theorem~\ref{thm:universal}.
\end{proof}

For the converse, we perturb one factor while leaving the others fixed.  The
product formula then promotes the perturbation to a measure equivalence of the
whole free product.

\begin{proof}[Proof of $H$ rigid $\Rightarrow$ each $H_i$ rigid]
Fix $j$.  If $H_j=1$ there is nothing to prove, since the only subgroup
isomorphic to $H_j$ is trivial and the identity witnesses rigidity; so assume
$H_j\neq1$.  Normalise as in \eqref{eq:normalised}; by
Proposition~\ref{prop:reduce}(i) it then suffices to prove that $H_j$ is
profinitely rigid in $A_j=\acl_F(H_j)$.

\medskip\noindent\emph{Perturbing one factor.}  Let $L_j\le A_j$ be finitely
generated and let $\alpha_j\colon H_j\to L_j$ be a measure equivalence; by
Lemma~\ref{lem:measure-stable} it is immaterial whether the measures are computed
in $A_j$ or in $F$.  Set
\[
  L_i:=H_i,\quad\alpha_i:=\mathrm{id}_{H_i}\ (i\neq j),\qquad
  L:=L_1*\dots*L_k,
\]
an internal free product inside $F$ because $L_i\le A_i$ for every $i$, and let
$\alpha\colon H\to L$ be the isomorphism determined by the $\alpha_i$ through the
universal property of the free product.  By Lemma~\ref{lem:product-measure},
applied once to $(H_i)_i$ and once to $(L_i)_i$, for every finite $G$ and every
$h\in\Hom(L,G)$,
\[
  \Pr\nolimits_{L,G}(h)=\prod_{i=1}^{k}\Pr\nolimits_{L_i,G}\bigl(h|_{L_i}\bigr)
  =\prod_{i=1}^{k}\Pr\nolimits_{H_i,G}\bigl(h|_{L_i}\circ\alpha_i\bigr)
  =\Pr\nolimits_{H,G}(h\circ\alpha),
\]
the middle equality holding for $i=j$ by hypothesis and for $i\neq j$ trivially.
Thus $\alpha$ is a measure equivalence $H\to L$, and profinite
rigidity of $H$ in $F$ supplies $\varphi\in\Aut(F)$ with $\varphi(H)=L$ and
$\varphi|_H=\alpha$.  In particular
\[
  \varphi(H_j)=\alpha(H_j)=L_j,\qquad \varphi|_{H_j}=\alpha_j .
\]

\medskip\noindent\emph{Descending to $A_j$.}  Automorphisms of $F$ preserve
algebraic closures, so $\varphi(A_j)=\varphi(\acl_F(H_j))=\acl_F(L_j)$.  Since
$L_j\le A_j$ and $A_j$ is a free factor of $F$, Lemma~\ref{lem:acl-trans} gives
$\acl_F(L_j)=\acl_{A_j}(L_j)$, a free factor of $A_j$.  Writing
$A_j=\varphi(A_j)*C$ and comparing ranks, $\rk\varphi(A_j)=\rk A_j<\infty$ forces
$\rk C=0$, so $\varphi(A_j)=A_j$.  Hence $\varphi|_{A_j}\in\Aut(A_j)$ satisfies
$\varphi|_{A_j}(H_j)=L_j$ and $\varphi|_{H_j}=\alpha_j$, which is exactly
profinite rigidity of $H_j$ in $A_j$.
\end{proof}

The forward implication uses the marking only to obtain equality on the
individual factors; after that point the argument is setwise.  It therefore has
the following setwise form.

\begin{proposition}\label{prop:free-prod-setwise}
In the setting of Theorem~\ref{thm:free-prod}, if every $H_i$ is setwise
profinitely rigid in $F$ then $H$ is setwise profinitely rigid in $F$, and hence
in every free extension $F*P$.
\end{proposition}

\begin{proof}
Repeat the proof that each $H_i$ rigid $\Rightarrow$ $H$ rigid, replacing
Corollary~\ref{cor:betai} by its setwise form: Lemma~\ref{lem:factorwise} gives
that $\alpha_i$ is a measure equivalence, and setwise rigidity of $H_i$ supplies
$\beta_i\in\Aut(F)$ with $\beta_i(H_i)=L_i$---the clause
$\beta_i|_{H_i}=\alpha_i$ is simply dropped.  Equation~\eqref{eq:star} uses only
$\beta_i(H_i)=L_i$, so Steps 2 and 3 go through verbatim and produce
$\varphi\in\Aut(F)$ with $\varphi|_{A_i}=\beta_i|_{A_i}$, whence
$\varphi(H_i)=L_i$ for every $i$ and therefore $\varphi(H)=L$.  The last
assertion is Theorem~\ref{thm:universal}(ii).
\end{proof}

We do not claim a converse to Proposition~\ref{prop:free-prod-setwise}.  The
marked converse uses the identity $\varphi|_H=\alpha$ to force
$\varphi(H_j)=L_j$ for the prescribed factor.  Setwise rigidity gives only
$\varphi(H)=L$; the factors may be permuted or mixed, so the same perturbation
argument does not recover rigidity of a specified $H_j$.  Thus the marked and
setwise converse problems are genuinely different, and the proof of
Theorem~\ref{thm:free-prod} does not settle the latter.

\begin{question}\label{q:setwise-freeprod-converse}
In the setting of Theorem~\ref{thm:free-prod}, if
$H=H_1*\dots*H_k$ is setwise profinitely rigid in $F$, must each $H_i$ be
setwise profinitely rigid in $F$?
\end{question}

\subsection{Examples and applications}\label{sec:new-families}

We now combine Theorem~\ref{thm:free-prod} with known rigid words.  The inputs
we use are primitive words \cite{PP,Wilton,GJ}, the words $x^d$ and
$[x,y]^d$ \cite{HMP}, and surface words \cite{Wilton2}. 

The free-product theorem can be packaged as a general amplification
principle for rigid words.

\begin{corollary}\label{cor:assembly}
Let
\[
  F=A_1*\dots*A_k*C
\]
be an internal free product decomposition of a finitely generated free group.
For each $i$, let $1\neq w_i\in A_i$ be a profinitely rigid word in $A_i$.  Then
\[
  H:=\langle w_1,\dots,w_k\rangle
    =\langle w_1\rangle*\dots*\langle w_k\rangle
\]
is profinitely rigid in $F$, and hence in every free extension $F*P$.
Moreover,
\[
  \acl_F(H)=\acl_{A_1}(\langle w_1\rangle)*\dots*
            \acl_{A_k}(\langle w_k\rangle).
\]
In particular, if every $w_i$ is algebraic in $A_i$, then
$\acl_F(H)=A_1*\dots*A_k$; if also $C=1$, then $H$ is algebraic in $F$.
\end{corollary}
\begin{proof}
By Corollary~\ref{cor:word-cyclic}, each cyclic subgroup $\langle w_i\rangle$
is profinitely rigid in $A_i$.  Theorem~\ref{thm:universal} promotes this to
profinite rigidity in $F$.  Apply Theorem~\ref{thm:free-prod} to the factors
$\langle w_i\rangle\le A_i$, adjoining the trivial subgroup in the factor $C$
when $C\neq1$.  This proves rigidity of $H$.  The formula for the algebraic
closure is Lemmas~\ref{lem:acl-trans} and~\ref{lem:acl-freeprod}; the final
claims are immediate.
\end{proof}

Any further family of profinitely rigid words can be inserted into the same
construction.  The present inputs already
produce rigid subgroups of unbounded rank.

\begin{corollary}\label{cor:examples}
Let $F$ be a finitely generated free group and $F=A_1*\dots*A_k*C$ an internal
free product decomposition in which, for each $i$, the subgroup $H_i\le A_i$ is
of one of the following types:
\begin{enumerate}
\item[(a)] $H_i=A_i$;
\item[(b)] $A_i=\langle u_i\rangle$ has rank one and
           $H_i=\langle u_i^{\,m_i}\rangle$ with
           $m_i\in\mathbb Z\setminus\{0\}$;
\item[(c)] $A_i=\langle u_i,v_i\rangle$ has rank two and
           $H_i=\bigl\langle[u_i,v_i]^{\,d_i}\bigr\rangle$ with
           $d_i\in\mathbb Z\setminus\{0\}$;
\item[(d)] $A_i=\langle u_{i,1},v_{i,1},\dots,u_{i,g},v_{i,g}\rangle$ has rank
           $2g$ and $H_i=\bigl\langle[u_{i,1},v_{i,1}]\cdots[u_{i,g},v_{i,g}]
           \bigr\rangle$;
\item[(e)] $A_i=\langle z_{i,1},\dots,z_{i,c}\rangle$ has rank $c$ and
           $H_i=\bigl\langle z_{i,1}^{2}\cdots z_{i,c}^{2}\bigr\rangle$.
\end{enumerate}
Then $H_1*\dots*H_k$ is profinitely rigid in $F$, and in every free extension
$F*P$ with $P$ finitely generated free.
\end{corollary}

\begin{proof}
Each $A_i$ is a free factor of $F$, so Theorem~\ref{thm:universal} allows us to
test rigidity of $H_i$ inside $A_i$.  In case (a) the claim follows from
\cite{PP}.  In cases (b) and (c), after identifying
$A_i$ with $F_1$ or $F_2$, respectively, the generator of $H_i$ becomes
$x^{m_i}$ or $[x,y]^{d_i}$; these words are rigid by \cite{HMP}, and
Corollary~\ref{cor:word-cyclic} gives rigidity of the corresponding cyclic
subgroups.  Cases (d) and (e) follow in the same way from \cite{AF, Wilton2}.

The assertion for free extensions follows from Theorem~\ref{thm:universal}.
\end{proof}

\begin{corollary}\label{cor:powers-basis}
Let $H=\langle x_1^{a_1},\dots,x_n^{a_n}\rangle\le F_n$, where $x_1,\dots,x_n$ is
a free basis of $F_n$ and $a_1,\dots,a_n$ are nonzero integers.  Then $H$ is
profinitely rigid in $F_n$, and hence profinitely rigid in every free extension
$F_n*P$ with $P$ finitely generated free.
\end{corollary}

\begin{proof}
Apply Corollary~\ref{cor:examples} with $A_i=\langle x_i\rangle$,
$H_i=\langle x_i^{a_i}\rangle$ and $C=1$; case (b) applies to every $i$. 
\end{proof}

\begin{proof}[Proof of Corollary~\ref{cor:powers}]
For $r\le n$, the subgroup $\langle x_1^{m_1},\dots,x_r^{m_r}\rangle$ lies in the
free factor $F_r=\langle x_1,\dots,x_r\rangle$ of $F_n$, where it is profinitely
rigid by Corollary~\ref{cor:powers-basis}.  By Theorem~\ref{thm:universal}, it
remains profinitely rigid in $F_n=F_r*\langle x_{r+1},\dots,x_n\rangle$.
\end{proof}

\begin{corollary}\label{cor:ranks}
Let $n\ge2$ and let $r$ be an integer with $1\le r\le n$.  Then
$F_n$ contains a profinitely rigid subgroup $H$ of rank $r$ with
$\acl_{F_n}(H)=F_n$, $H\neq F_n$ and $[F_n:H]=\infty$.  It is profinitely rigid
in every free extension $F_n*P$.
\end{corollary}

\begin{proof}
Write $n=n_1+\dots+n_r$ with every $n_i\ge1$; this is possible because
$1\le r\le n$.  Split a free basis of $F_n$ accordingly and take the internal
decomposition
\[
  F_n=A_1*\dots*A_r,\qquad
  A_i=\langle z_{i,1},\dots,z_{i,n_i}\rangle,
\]
with
\[
  H_i=\bigl\langle w_i\bigr\rangle,\qquad
  w_i:=z_{i,1}^{2}z_{i,2}^{2}\cdots z_{i,n_i}^{2}.
\]
Each $H_i$ is of type (e) in Corollary~\ref{cor:examples}, so
$H=H_1*\dots*H_r$ is profinitely rigid in $F_n$ and in every $F_n*P$.  Its rank
is $r$, since the $w_i$ freely generate $H$.

Next, each $H_i$ is algebraic in $A_i$.  For $n_i=1$, this says that the
non-trivial subgroup $\langle z_{i,1}^{2}\rangle$ of $A_i\cong\mathbb Z$ lies in
no proper free factor, which is clear since the only proper free factor of
$\mathbb Z$ is trivial.  For $n_i\ge2$, suppose $A_i=B*D$ with $w_i\in B$ and
$D\neq1$.  Then
\[
  A_i/\langle\langle w_i\rangle\rangle
  \cong\bigl(B/\langle\langle w_i\rangle\rangle_B\bigr)*D,
\]
a free product with $D\neq1$.  The other factor is nontrivial as well.  Indeed,
if $B/\langle\langle w_i\rangle\rangle_B=1$, then the abelianization of $B$ is
cyclic, so $B$ has rank one; in the resulting infinite cyclic group, $w_i$
normally generates only when it is a generator.  This would make $w_i$
primitive in $A_i$, contradicting the fact that its abelianization
$2(e_1+\cdots+e_{n_i})$ is not primitive in $\mathbb Z^{n_i}$.  A free product
of two nontrivial groups has more than one end.  But
$A_i/\langle\langle w_i\rangle\rangle$ is the non-orientable surface group
$N_{n_i}=\langle z_1,\dots,z_{n_i}\mid z_1^{2}\cdots z_{n_i}^{2}\rangle$, which
for $n_i=2$ is the Klein bottle group, virtually $\mathbb Z^{2}$ and hence
one-ended, and for $n_i\ge3$ is a closed hyperbolic surface group, again
one-ended.  This is a contradiction, so $H_i$ is algebraic in $A_i$.  By
Lemma~\ref{lem:acl-freeprod},
\[
  \acl_{F_n}(H)=\acl_{F_n}(H_1)*\dots*\acl_{F_n}(H_r)=A_1*\dots*A_r=F_n .
\]

If $r<n$ then $H\neq F_n$ on rank grounds.  If $r=n$ then every $n_i=1$ and
$H=\langle z_1^{2},\dots,z_n^{2}\rangle$, which is contained in the proper
subgroup $F_n^{\,2}[F_n,F_n]$ and so is not $F_n$.  Finally, if $H$ had finite
index $m$ in $F_n$ then Nielsen--Schreier would give $r=\rk H=1+m(n-1)$; since
$r\le n$ and $n\ge2$ this forces $m(n-1)\le n-1$, hence $m=1$ and $H=F_n$, a
contradiction.  So $[F_n:H]=\infty$.
\end{proof}

\subsection{Further questions: transitivity and the passage back to words}
\label{sec:back-to-words}

Theorem~\ref{thm:free-prod} converts rigid words placed in distinct free factors
into rigid subgroups, but it does not directly produce new rigid words: as soon
as at least two cyclic factors are non-trivial, their free product has rank at
least two, so Corollary~\ref{cor:word-cyclic} no longer applies. A natural
reverse mechanism would therefore be transitivity of rigidity along an
intermediate subgroup. This is not formal, because the measures defining
rigidity of $K$ in $H$ are induced from $\Hom(H,G)$, whereas those defining
rigidity of $K$ in $F$ are induced from $\Hom(F,G)$.

\begin{question}\label{q:transitivity}
Let $K\le H\le F$, where $F$ is a finitely generated free group and $H,K$ are
finitely generated. If $K$ is profinitely rigid in $H$ and $H$ is profinitely
rigid in $F$, must $K$ be profinitely rigid in $F$?
\end{question}

\begin{proposition}\label{prop:acl-intermediate}
Let $K\le H\le F$, where $F$ is a finitely generated free group and $H,K$ are
finitely generated, and suppose $\acl_F(K)\le H$. Then
\[
  \acl_H(K)=\acl_F(K).
\]
Consequently $K$ is profinitely rigid in $H$ if and only if it is profinitely
rigid in $F$. The same equivalence holds for setwise profinite rigidity.
\end{proposition}

\begin{proof}
Put $A=\acl_F(K)$. Since $A$ is a free factor of $F$ and
$A\le H\le F$, write $F=A*C$ and apply the Kurosh subgroup theorem to
$H\le A*C$. The factor corresponding to the double coset $H\cdot1\cdot A$ is
$H\cap A=A$, so $A$ is a free factor of $H$. Hence
$B:=\acl_H(K)\le A$.

Write $H=B*D$. Applying Kurosh again, now to $A\le H=B*D$, shows that
$B=A\cap B$ is a free factor of $A$. But $K$ is algebraic in
$A=\acl_F(K)$, and $B$ is a free factor of $A$ containing $K$; therefore
$B=A$. Thus $\acl_H(K)=\acl_F(K)$.

Apply Theorem~\ref{thm:universal} to $K$ inside $F$ and inside $H$. In either
the marked or setwise setting, rigidity is equivalent to rigidity in the
corresponding algebraic closure, which is $A$ in both ambient groups.
\end{proof}

\begin{corollary}\label{cor:conditional-surface}
Assume Question~\ref{q:transitivity} has a positive answer. Let $a,b\ge0$ with
$a+b\ge1$, and let all $d_i,m_j$ be non-zero integers. Then
\[
  [x_1,y_1]^{d_1}\cdots[x_a,y_a]^{d_a}
  z_1^{\,m_1}\cdots z_b^{\,m_b}\in F_{2a+b}
\]
is profinitely rigid. In particular, this would give profinite rigidity of all
words $x_1^{m_1}\cdots x_b^{m_b}$ and
$[x_1,y_1]^{d_1}\cdots[x_a,y_a]^{d_a}$ with non-zero exponents.
\end{corollary}

\begin{proof}
Let
\[
  H=\mathop{\ast}_{i=1}^{a}\langle [x_i,y_i]^{d_i}\rangle
    *\mathop{\ast}_{j=1}^{b}\langle z_j^{m_j}\rangle
  \le F_{2a+b}.
\]
Each cyclic factor is profinitely rigid, by the rigidity of
$[x,y]^d$ and $z^m$ from \cite{HMP}, and Corollary~\ref{cor:assembly} therefore makes $H$
profinitely rigid in $F_{2a+b}$. If $h_1,\dots,h_{a+b}$ are the displayed free
generators of $H$, then $h_1\cdots h_{a+b}$ is primitive in $H$: it is obtained
from $h_1$ by a Nielsen transformation. Hence
$\langle h_1\cdots h_{a+b}\rangle$ is a free factor of $H$ and is profinitely
rigid there. Question~\ref{q:transitivity} would then make this cyclic subgroup
profinitely rigid in $F_{2a+b}$, and Corollary~\ref{cor:word-cyclic} gives the
claim for the displayed word.
\end{proof}

\section{Tuples and induced distributions}\label{sec:tuples}

We conclude with the tuple formulation.  It is logically independent of the
preceding arguments, but it identifies marked subgroup rigidity with rigidity of
any generating tuple and gives a profinite characterization of tuple
measure-equivalence.

Let $\Gamma$ be a finitely generated group.  For a finite group $G$, denote by
$\Hom(\Gamma,G)$ the set of homomorphisms $\Gamma\to G$ and by
$\Epi(\Gamma,G)$ the subset of surjective homomorphisms.  Given a $k$-tuple
$T=(\gamma_1,\dots,\gamma_k)\in\Gamma^{k}$, the evaluation map
\[
  \mathrm{ev}_T\colon\Hom(\Gamma,G)\longrightarrow G^{k},\qquad
  \varphi\mapsto(\varphi(\gamma_1),\dots,\varphi(\gamma_k)),
\]
pushes forward the uniform measure on $\Hom(\Gamma,G)$ to a probability measure
on $G^{k}$.  Concretely, for $g=(g_1,\dots,g_k)\in G^{k}$ define the evaluation
count
\[
  N_T(G;g):=\bigl|\{\varphi\in\Hom(\Gamma,G)\mid \varphi(\gamma_i)=g_i
  \text{ for all }i\}\bigr| .
\]
When $k=1$ and $T=(w)$ for a word $w\in F_n$, this recovers the word measure of
Section~\ref{sec:words}: the probability that $w_G$ takes a prescribed value
under a uniformly random homomorphism $F_n\to G$.

Two $k$-tuples $T,T'\in\Gamma^{k}$ \emph{induce the same measure} if
\[
  N_T(G;g)=N_{T'}(G;g)
\]
for every finite group $G$ and every $g\in G^{k}$.

\begin{definition}\label{def:tuple-rigid}
A $k$-tuple $T\in\Gamma^{k}$ is \emph{profinitely rigid in $\Gamma$} if every
$T'\in\Gamma^{k}$ inducing the same measure as $T$ on every finite group lies in
the same $\Aut(\Gamma)$-orbit as $T$, that is, $T'=\varphi(T)$ for some
$\varphi\in\Aut(\Gamma)$ acting coordinatewise.
\end{definition}

For $k=1$ and $\Gamma=F_n$, Definition~\ref{def:tuple-rigid} is exactly the
usual rigidity problem for words.  Automorphic equivalence of tuples in a free
group is decidable by Whitehead's algorithm \cite{Whitehead}; the question is
whether the finite-group measures determine the same orbits.

\subsection{A profinite characterisation of measure-equivalence}
\label{sec:tuple-char}

The following theorem extends \cite[Theorem 2.2]{HMP} from words to arbitrary
finite tuples.  We begin with notation.  For $T\in\Gamma^{k}$ and $g\in G^{k}$ set
\[
  \Hom_{(T,g)}(\Gamma,G):=\{\varphi\in\Hom(\Gamma,G)\mid\varphi(\gamma_i)=g_i\ \forall i\},
\]
\[
  \Epi_{(T,g)}(\Gamma,G):=\{\varphi\in\Epi(\Gamma,G)\mid\varphi(\gamma_i)=g_i\ \forall i\},
\]
and
\[
  \ImEpi_T(\Gamma,G):=\{\varphi(T)\in G^{k}\mid\varphi\in\Epi(\Gamma,G)\}.
\]
For a finite group $G$, define the normal subgroup
\[
  K_\Gamma(G):=\bigcap_{\varphi\in\Epi(\Gamma,G)}\ker(\varphi),
\]
with the convention $K_\Gamma(G)=\Gamma$ when $\Epi(\Gamma,G)=\varnothing$, and
set $Q_G:=\Gamma/K_\Gamma(G)$.  Since $\Epi(\Gamma,G)$ is finite, the map
$\gamma\mapsto(\varphi(\gamma))_{\varphi\in\Epi(\Gamma,G)}$ embeds $Q_G$ into the
finite power $G^{|\Epi(\Gamma,G)|}$; in particular $Q_G$ is finite.  Moreover
$K_\Gamma(G)$ is characteristic in $\Gamma$, since $\Aut(\Gamma)$ permutes
$\Epi(\Gamma,G)$.  Every finite quotient of $\Gamma$ isomorphic to $G$ has the
form $\Gamma/N$ with $K_\Gamma(G)\le N$, so the quotients $\{Q_G\}_G$ are
cofinal in the inverse system of all finite quotients of $\Gamma$.

\begin{theorem}\label{thm:tuple-char}
Let $\Gamma$ be a finitely generated group and let $T=(\gamma_1,\dots,\gamma_k)$
and $T'=(\delta_1,\dots,\delta_k)$ be elements of $\Gamma^{k}$.  Write
$\gamma_i,\delta_i$ also for their images in $\widehat\Gamma$ under $\iota$.
The following are equivalent.
\begin{enumerate}
\item There is a continuous automorphism $\Theta\in\Aut(\widehat\Gamma)$ with
      $\Theta(\gamma_i)=\delta_i$ for all $i=1,\dots,k$.
\item For every finite group $G$, the pushforward measures on $G^{k}$ induced by
      $T$ and $T'$ coincide.
\item For every finite group $G$ and every $g\in G^{k}$,
      $\bigl|\Epi_{(T,g)}(\Gamma,G)\bigr|=\bigl|\Epi_{(T',g)}(\Gamma,G)\bigr|$.
\item For every finite group $G$,
      $\ImEpi_T(\Gamma,G)=\ImEpi_{T'}(\Gamma,G)$ as subsets of $G^{k}$.
\item For every finite group $G$ there is an automorphism $f$ of
      $Q_G=\Gamma/K_\Gamma(G)$ such that
      $f(\gamma_iK_\Gamma(G))=\delta_iK_\Gamma(G)$ for all $i$.
\end{enumerate}
\end{theorem}

\begin{proof}
    The proof is similar to the proof of \cite[Theorem~2.2]{HMP}.
\end{proof}

In particular, two tuples in $F_n$ induce the same finite-group measures if
and only if they lie in the same $\Aut(\widehat{F_n})$-orbit.

\subsection{Tuples versus subgroups}\label{sec:tuple-vs-subgroup}

We now compare tuple rigidity with marked subgroup rigidity.  The forward
passage from a rigid tuple to its generated subgroup uses only finite
generation; residual finiteness is needed in the converse to recover relations
from the profinite completion.

\begin{proof}[Proof of Theorem~\ref{thm:tuples}]
Let $H=\langle w_1,\dots,w_k\rangle\le\Gamma$ and write $T=(w_1,\dots,w_k)$.

$(\Leftarrow)$  Suppose the tuple $T$ is profinitely rigid.  Let $L\le\Gamma$ be
finitely generated and $\alpha\colon H\to L$ an isomorphism identifying the
induced distributions.  Put $u_i=\alpha(w_i)\in L$, so that $(u_1,\dots,u_k)$
generates $L$.

We claim that $T=(w_1,\dots,w_k)$ and $T'=(u_1,\dots,u_k)$ induce the same
measure on every finite group $G$.  Fix $g=(g_1,\dots,g_k)\in G^{k}$.  If the
assignment $w_i\mapsto g_i$ extends to a homomorphism $h_g\colon H\to G$ then,
since the $w_i$ generate $H$,
\[
  N_T(G;g)=\bigl|\{\varphi\in\Hom(\Gamma,G)\mid\varphi|_H=h_g\}\bigr|
  =|\Hom(\Gamma,G)|\cdot\Pr\nolimits_{H,G}(h_g);
\]
otherwise $N_T(G;g)=0$.  The assignment $u_i\mapsto g_i$ extends to a
homomorphism $h'_g\colon L\to G$ if and only if $h_g$ exists, in which case
$h'_g=h_g\circ\alpha^{-1}$; and the hypothesis on $\alpha$ gives
$\Pr_{H,G}(h_g)=\Pr_{L,G}(h'_g)$.  Hence $N_T(G;g)=N_{T'}(G;g)$, proving the
claim.  By profinite rigidity of $T$ there is $\varphi\in\Aut(\Gamma)$ with
$\varphi(w_i)=u_i=\alpha(w_i)$ for all $i$, so $\varphi|_H=\alpha$ and
$\varphi(H)=L$.  Therefore $H$ is profinitely rigid.  This direction uses no
hypothesis on $\Gamma$ beyond finite generation.

$(\Rightarrow)$  Suppose $H$ is profinitely rigid.  Let
$(u_1,\dots,u_k)\in\Gamma^{k}$ induce the same measures as $T$ on all finite
groups, and set $H'=\langle u_1,\dots,u_k\rangle$.  By
Theorem~\ref{thm:tuple-char} there is $\Theta\in\Aut(\widehat\Gamma)$ with
$\Theta(\iota(w_i))=\iota(u_i)$ for all $i$.  Every relation
$r(w_1,\dots,w_k)=1$ holding in $H$ then yields
\[
  \iota\bigl(r(u_1,\dots,u_k)\bigr)
  =\Theta\bigl(\iota(r(w_1,\dots,w_k))\bigr)=1,
\]
and since $\Gamma$ is residually finite $\iota$ is injective, so
$r(u_1,\dots,u_k)=1$; applying the same argument to $\Theta^{-1}$ gives the
  converse implication.  Hence $w_i\mapsto u_i$ extends to a group isomorphism
  $\alpha_0\colon H\xrightarrow{\ \sim\ }H'$.  The argument of the previous
  paragraph, read in reverse, shows that $H$ and $H'$ induce the same subgroup
  measures via $\alpha_0$: the mass of $h\in\Hom(H,G)$ is the normalized
  evaluation count at $(h(w_1),\dots,h(w_k))$, and the corresponding count for
  $h\circ\alpha_0^{-1}$ is equal by hypothesis.  By profinite rigidity of $H$ there is
$\varphi\in\Aut(\Gamma)$ with $\varphi|_H=\alpha_0$, so $\varphi(w_i)=u_i$ for
all $i$.  Hence the two tuples are automorphic in $\Gamma$.  Residual finiteness
was used only in this paragraph.
\end{proof}

The case $k=1$ recovers Corollary~\ref{cor:word-cyclic}, so Shalev's
conjecture is precisely the rank-one case of marked subgroup rigidity.  The
same dictionary transfers the preceding subgroup results to generating tuples;
for example:

\begin{corollary}\label{cor:freefactor-tuple}
If $(w_1,\dots,w_k)$ is a tuple in $F_n$ whose generated subgroup is a free
factor of $F_n$, then the tuple is profinitely rigid.
\end{corollary}

\begin{proof}
Free factors of $F_n$ are profinitely rigid \cite{PP}, so the claim follows from
Theorem~\ref{thm:tuples}.
\end{proof}

Likewise, every rigid subgroup constructed in
Section~\ref{sec:new-families} yields rigid generating tuples of the
corresponding rank.

\section*{Acknowledgements}

The author acknowledges support from the Department of Atomic Energy,
Government of India, under project no.~RTI4019.


\begin{thebibliography}{99}


\bibitem{AmitVishne} A. Amit and U. Vishne,
\emph{Characters and solutions to equations in finite groups},
J. Algebra Appl. \textbf{10} (2011), no.~4, 675--686.

\bibitem{AF} D. Ascari and J. Fruchter,
\emph{Virtual homological torsion in graphs of free groups with cyclic
edge groups}, preprint, arXiv:2505.20960 (2025).

\bibitem{GJ} A. Garrido and A. Jaikin-Zapirain,
\emph{Free factors and profinite completions},
Int. Math. Res. Not. IMRN (2023), no.~24, 21320--21345.

\bibitem{Hall} M. Hall, Jr.,
\emph{Coset representations in free groups},
Trans. Amer. Math. Soc. \textbf{67} (1949), no.~2, 421--432.

\bibitem{HMP} L. Hanany, C. Meiri, and D. Puder,
\emph{Some orbits of free words that are determined by measures on
finite groups}, J. Algebra \textbf{555} (2020), 305--324.

\bibitem{JSZ} A. Jaikin-Zapirain, H. Souza, and P. Zalesski,
\emph{Profinite detection of free products and free factors},
preprint, arXiv:2603.16674 (2026).

\bibitem{MVW} A. Miasnikov, E. Ventura, and P. Weil,
\emph{Algebraic extensions in free groups},
in Geometric Group Theory, Trends Math.,
Birkh\"auser, Basel, 2007, 225--253.

\bibitem{PP} D. Puder and O. Parzanchevski,
\emph{Measure preserving words are primitive},
J. Amer. Math. Soc. \textbf{28} (2015), no.~1, 63--97.

\bibitem{RZ} L. Ribes and P. Zalesskii,
\emph{Profinite Groups}, 2nd ed.,
Ergeb. Math. Grenzgeb. (3), vol.~40,
Springer, Berlin, 2010.

\bibitem{Shalev} A. Shalev,
\emph{Some results and problems in the theory of word maps},
in Erd\H{o}s Centennial
(L. Lov\'asz, I. Z. Ruzsa, and V. T. S\'os, eds.),
Bolyai Soc. Math. Stud., vol.~25,
Springer, Berlin--Heidelberg, 2013, 611--649.

\bibitem{Whitehead} J. H. C. Whitehead,
\emph{On equivalent sets of elements in a free group},
Ann. of Math. (2) \textbf{37} (1936), no.~4, 782--800.

\bibitem{Wilton} H. Wilton,
\emph{Essential surfaces in graph pairs},
J. Amer. Math. Soc. \textbf{31} (2018), no.~4, 893--919.


\bibitem{Wilton2} H. Wilton,
\emph{On the profinite rigidity of surface groups and surface words},
C. R. Math. Acad. Sci. Paris \textbf{359} (2021), no.~2, 119--122.

\end{thebibliography}
\end{document}